\documentclass[a4paper,12pt,english,reqno]{article}

\title{Existence and uniqueness for a class of fractional drift-diffusion equations}
\author{Thomas Hudson and Matthaeus Ragg}
\date{February 2025}

\usepackage[utf8]{inputenc}
\usepackage[english]{babel}
\usepackage{amsmath,amsfonts,amssymb,amsthm,mathtools}
\usepackage{cite}
\usepackage{color}
\usepackage{graphicx}
\usepackage{subfigure}
\usepackage{enumerate}
\usepackage{pdfsync}
\numberwithin{equation}{section}
\usepackage[a4paper,margin=2cm]{geometry}

\usepackage{hyperref}

\def\d{\,\mathrm{d}}
\def\R{\mathbb{R}}
\def\T{\mathbb{T}}
\def\Z{\mathbb{Z}}

\def\e{\mathrm{e}}

\newcommand*{\pd}[3][\!\!\;]{\frac{\partial^{#1} #2}{\partial #3^{#1}}}
\newcommand*{\deriv}[3][\!\!\;]{\frac{\mathrm{d}^{#1} #2}{\mathrm{d} #3^{#1}}}
\newcommand*{\flap}[1]{(-\Delta)^{{#1}}}
\def\F{\mathcal{F}}

\def\mspan{\mathrm{span}}
\def\Z{\mathbb{Z}}
\def\C{\mathbb{C}}

\newcommand*{\sobsp}[2][]{\mathcal{H}_{#1}^{#2}}
\newcommand*{\inprod}[3][L^2]{\left\langle #2 , #3 \right\rangle_{#1}}
\newcommand*{\norm}[2][L^2]{\left\| #2 \right\|_{#1}}
\def\*#1{\mathbf{#1}}
\def\ft#1{\hat{#1}}
\def\longft#1{\mathcal{F}\left[#1\right]}

\def\Xint#1{\mathchoice
{\XXint\displaystyle\textstyle{#1}}%
{\XXint\textstyle\scriptstyle{#1}}%
{\XXint\scriptstyle\scriptscriptstyle{#1}}%
{\XXint\scriptscriptstyle\scriptscriptstyle{#1}}%
\!\int}
\def\XXint#1#2#3{{\setbox0=\hbox{$#1{#2#3}{\int}$ }
\vcenter{\hbox{$#2#3$ }}\kern-.6\wd0}}

\def\dashint{\Xint-}
\def\e{\mathrm{e}}

\DeclareMathOperator*{\esssup}{ess\,sup}

\usepackage[capitalize,nameinlink]{cleveref}

\newtheorem{theorem}{Theorem}[section]
\newtheorem{lemma}[theorem]{Lemma}
\newtheorem{proposition}[theorem]{Proposition}

\theoremstyle{definition}
\newtheorem{definition}[theorem]{Definition}
\newtheorem{remark}[theorem]{Remark}

\begin{document}

\maketitle
\tableofcontents
\begin{abstract}
    This work establishes the existence and uniqueness of solutions to the fractional diffusion equation 
    \begin{equation*}
        \pd[\alpha]{u}{t} + K\flap{\beta} u - \nabla \cdot (\nabla V u) = f
    \end{equation*}
    on a $d$-dimensional torus, subject to sufficient conditions on the input parameters.
    The focus is on fractional orders $\alpha$ and $\beta$ less than 1.
    The strategy uses a Galerkin method and focuses on the additional complexity that comes from the fractional-order derivatives.
    Additional Sobolev regularity of the solution is shown.
    The spectral approach to the existence proof suggests an algorithm to compute explicit solutions numerically,
    and the regularity results are used to support a rigorous convergence analysis of the proposed numerical scheme.
\end{abstract}

\newpage
\section{Introduction}

This paper studies the well-posedness of the fractional partial differential equation (FPDE)
\begin{equation}\label{eq:original_pde}
    \pd[\alpha]{u}{t}+K(-\Delta)^{\beta}u - \nabla \cdot\big(\nabla V u\big)=0
\end{equation}
for $\alpha,\beta\in(0,1]$. In the case where $\alpha=\beta=1$, this is precisely the Fokker--Planck equation for a drift-diffusion process driven by Brownian motion. Here, we go beyond this classic setting, considering a generalisation to cases in which the equation contains fractional derivatives in both time and space. In particular, the spatial and temporal fractional derivatives involved in the equation above are taken to be of slightly different characters: we will explain more precisely what each of these derivatives mean in detail below.

To simplify issues with boundary conditions, we consider the equation in a periodic setting, so that $u:\T^d\times[0,T]\to\R$ with $\T = \R/(2\pi\Z)$, a periodic domain of length $2\pi$. The function $V : \T^d \to \R$ is a potential which we will assume to be sufficiently regular.

This FPDE involves both a Caputo fractional time derivative of order $\alpha$ and a fractional power $\beta$ of the Laplacian. These can be thought of as modifications to a first-order derivative and a Laplacian, respectively; precise definitions are given in \cref{sec:prelims}.
The focus of this paper is to establish existence, uniqueness and regularity results for the solution of a more general version of \eqref{eq:original_pde}, as well as the convergence of a corresponding numerical scheme.

If $\alpha = \beta = 1$, then the Caputo fractional derivative in time corresponds to the standard integer-order derivative, and the fractional Laplacian becomes a standard Laplacian. As such, we see that the family of equations considered generalises the usual integer-order drift-diffusion equation, where the first two terms model the diffusion of a particle or the heat flow over time.
The Caputo fractional derivative encompasses a memory property in the dynamics: when $\alpha<1$, the evolution of the solution over time depends globally on all previous states (see \cite{podlubny1999}). Such memory properties occur when the diffusive process exhibits time correlations.

On the other hand, the fractional Laplacian of order $\beta$ models a particle diffusing according to a $2\beta$-stable L\'evy process instead of Brownian motion (compare \cite{lischke2020}). In this case, the particle is assumed to make infinitesimal diffusive jumps sampled from a suitably rescaled version of the probability distribution shown in \cref{fig:Levy_process}.
\begin{figure}[h]
    \centering
    \includegraphics[width=.8\linewidth]{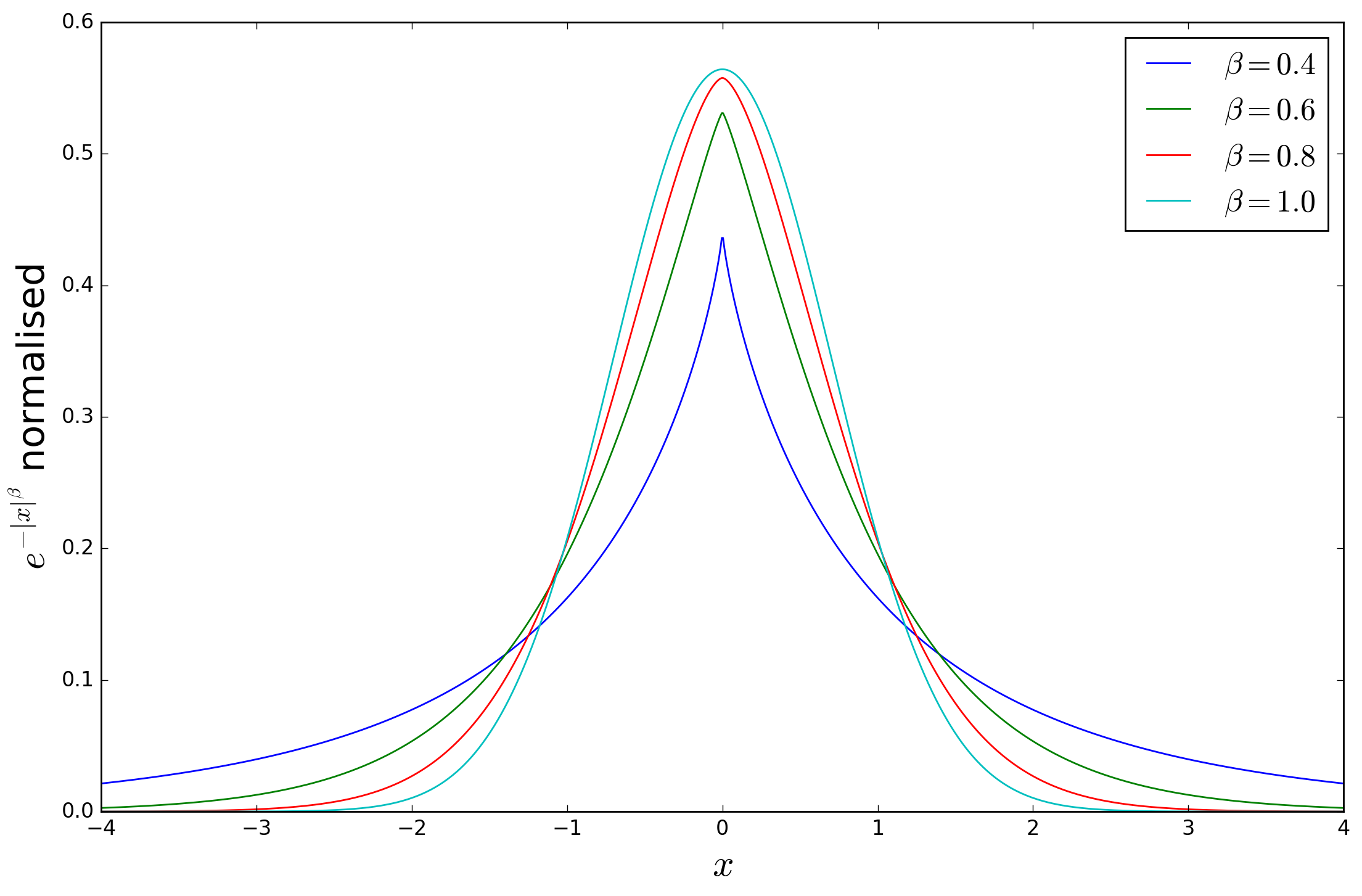}
    \caption{Distributions underlying the $2\beta$-stable L\'evy process. The case where $\beta=1$ corresponds to a normal distribution.}
    \label{fig:Levy_process}
\end{figure}

\begin{figure}
    \centering
    \includegraphics[width=.8\linewidth]{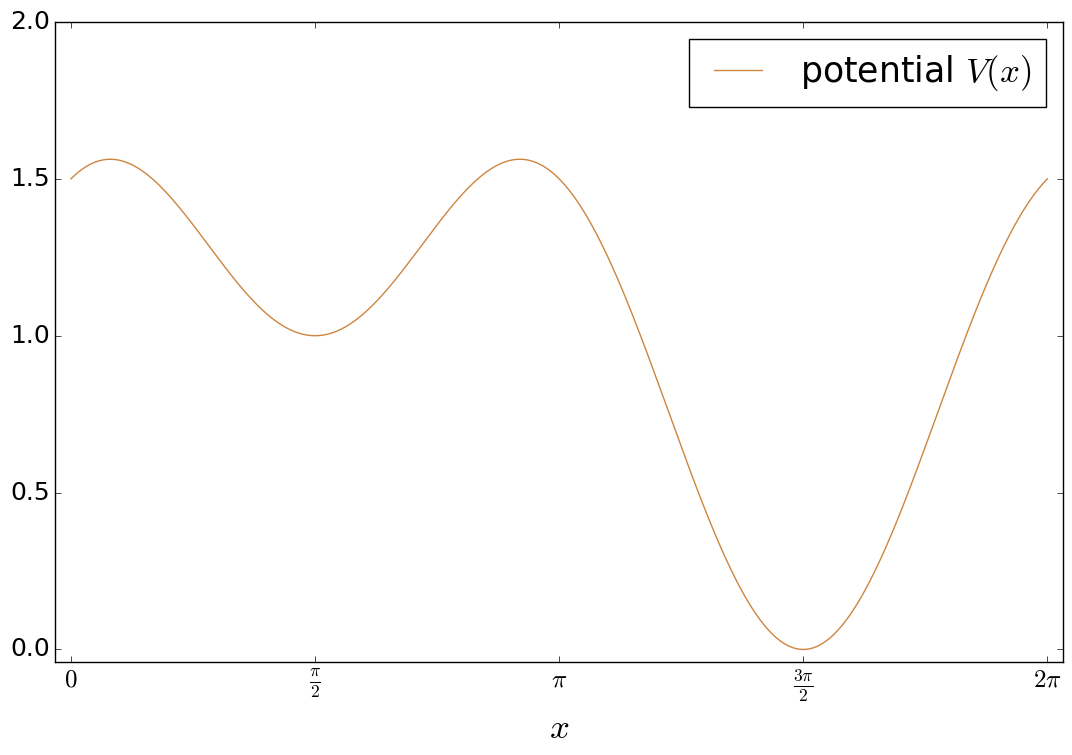}
    \caption{An example of a smooth potential $V$. The example was used to compute the numerical solutions shown in the figures which follow.}
    \label{fig:potential}
\end{figure}

The drift term in \eqref{eq:original_pde} involves a potential $V$ that the particle tries to minimise.
To gain some intuition about the action of this term, consider the one-dimensional potential shown in \cref{fig:potential}. We expect that, as $t \to \infty$, a particle prefers to sit inside the wells of the potential around $\frac\pi2$ and $\frac{3\pi}2$; indeed, if $K=0$ in \eqref{eq:original_pde}, then solutions to the resulting transport equation would concentrate around these points.

The distribution functions \cref{fig:Levy_process} suggests that for smaller $\beta$, the particle is more likely to make large jumps, as the tails of the distribution are heavier. In turn, this suggests that the particle may be more likely to escape from the minima of the potential $V$ and escape into a higher-energy metastable states. It turns out this intuition is indeed borne out by numerical simulation: in \cref{fig:final_time_sol}, we illustrate how the diffusion process affects the stable-state solution $u_\infty$. \cref{fig:3d_sol} shows the evolution of the solution from an initial condition, assuming the particle is initially uniformly distributed close to the less favourable local minimum.
\begin{figure}
    \centering
    \includegraphics[width=.8\linewidth]{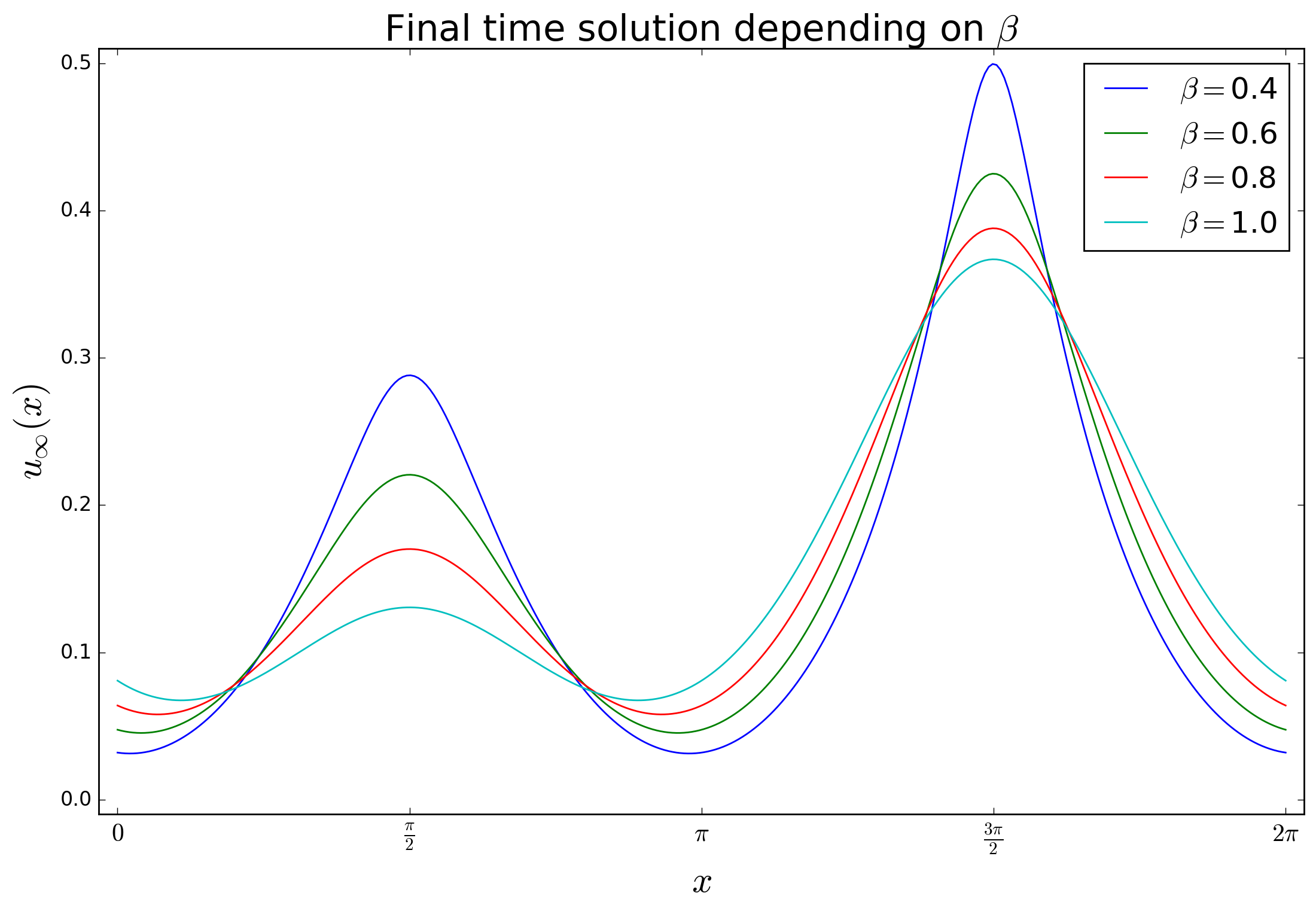}
    \caption{The approximate solution at final time as $\beta$ varies. To compute the solution we used $\alpha=1$ and a final time of $T=10000$. Smaller $\beta$ entails fatter tails of the noise distribution, and we see the solution is more peaked in this case, but also has more mass around the metastable state close to $\tfrac\pi2$.}
    \label{fig:final_time_sol}
\end{figure}

\begin{figure}
    \centering
    \includegraphics[width=1\linewidth]{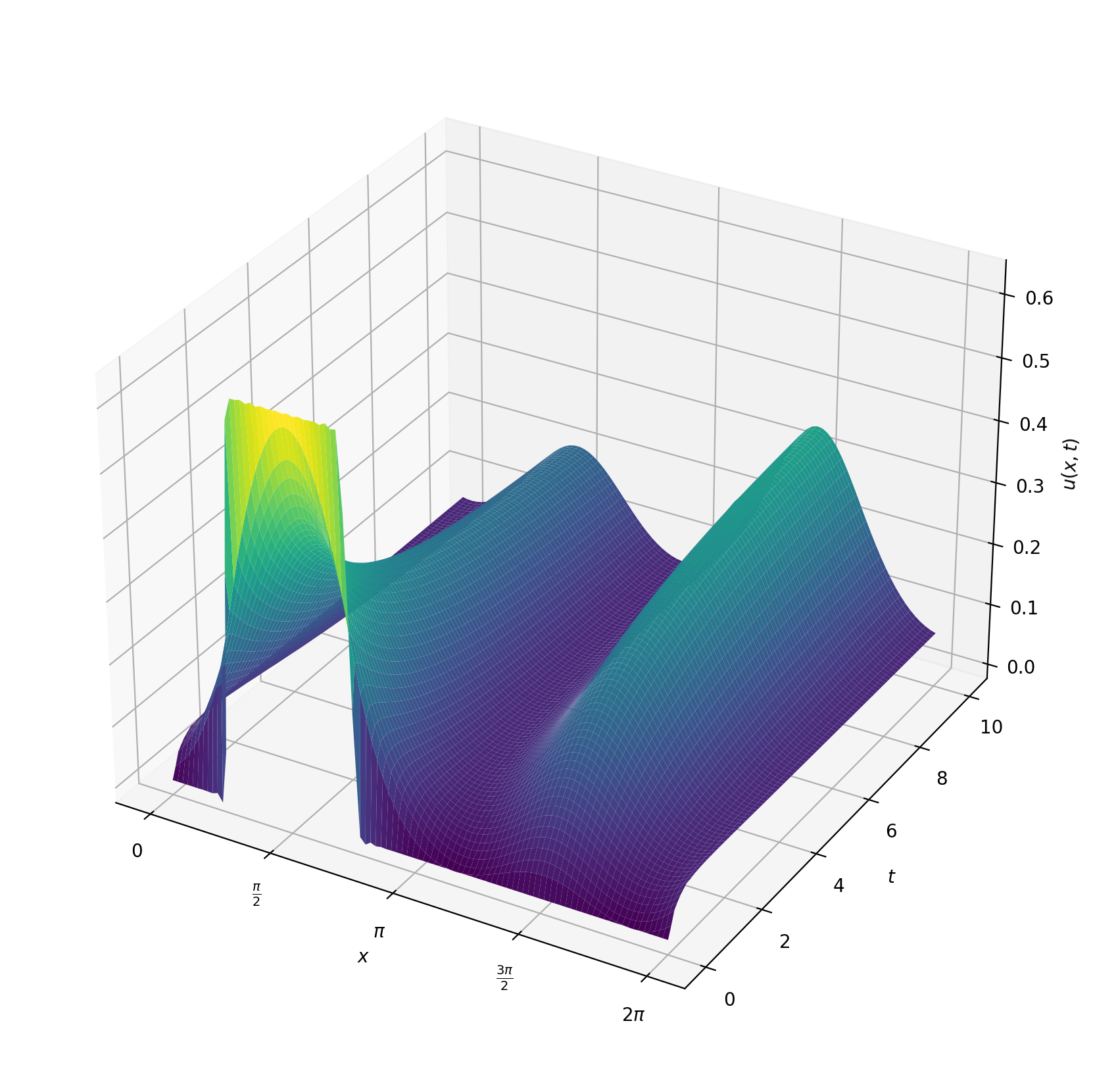}
    \caption{A numerical approximation of the solution for the case $\alpha = \beta = 0.8$, where the initial condition is a characteristic function centred at $\tfrac\pi2$.}
    \label{fig:3d_sol}
\end{figure}

\def\poincare{\hyperref[prop:poincare]{Poincar\'e's inequality} }

\subsection{Main results}
The first of our main results is to establish well-posedness of a mild generalisation of \eqref{eq:original_pde}, where we allow the potential $V$ to depend on time and introduce a source term $f$ on the right-hand side of the equation, allowing sources and sinks. In this context, we aim to find a weak solution with the initial condition given by $u(x, 0) = h(x)$ for some given $h$.
In other words, we seek $u : \T^d \times [0,T] \to \R$ satisfying
\begin{equation}\begin{aligned}\label{eq:weak_FPDE_0}
    \int_{\T^d} \left( \pd[\alpha]{u}{t}+K(-\Delta)^{\beta}u - \nabla \cdot\big(\nabla V u\big) \right) v \d x &= \int_{\T^d} f v \d x &&\text{ in } \T^d\times (0,T) \\
    u &= h  &&\text{ on } \T^d \times \{0\}
\end{aligned}\end{equation}
for every $v \in L^2(\T^d)$.
Here, $K > 0$ is a positive constant, $V : \T^d \times [0,T] \to \R$ a time-dependent potential and $f : \T^d \times [0,T] \to \R$ is a source term, and we will require $\alpha \in (0,1]$ and $\beta > \frac12$. Later, we will assume that 
\begin{equation}\label{eq:source_term_condition}
    \int_{\T^d} f(x,t) \d x = 0
\end{equation}
for almost every $t\in[0,T]$. Intuitively-speaking, this means that $f$ only redistributes mass, but does not add or remove any, and is a standard assumption required to make sense of the equation in the periodic setting.
This requirement is used to justify \cref{rmk:norms}, so that we do not have to worry about the integral mean term in the statement of \poincare in the following analysis.
However, the assumption \eqref{eq:source_term_condition} can be dropped and replaced by an estimate on the growth of $|(u)|$ in time depending on $f$, if one is willing to include an additional term in all the following computations.

The main goal of this work is therefore to show the following theorem:
\begin{theorem}
    Suppose $\beta > \frac12$, and $V$, $f$ and $h$ are sufficiently regular.
    Then there is a unique solution $u:\T^d \times [0,T] \to \R$ satisfying \eqref{eq:weak_FPDE_0}. Furthermore, the solution $u$ is sufficiently regular to make sense of $\pd[\alpha]{u}{t}$, $\flap{\beta}u$ and $\nabla u$ for almost every time $t$.
\end{theorem}
\noindent What ``sufficiently regular'' in each case means is made precise in the statements of \cref{thm:existence} and \cref{thm:regularity}. As an example, the result holds if $f$ is bounded, $\flap{\beta/2}h$ exists, $V$ is three times continuously differentiable in space and its first spatial derivative is continuous in time. Theorems \ref{thm:existence} and \ref{thm:regularity} specify slightly weaker assumptions.

The methods we use depend centrally on the assumption $\beta > \frac12$ and the fact that the equation is considered on a periodic domain $\T^d$.
The former assumption ensures that the order of differentiation of the fractional Laplacian dominates the forcing term involving $V$, so that the equation remains of parabolic type. The latter assumption allows us to work with Fourier series to simplify many expressions. 
However, we do observe that some statements for $\beta = \frac12$ are also valid, and the numerical scheme we use appears to converge to a reasonable solution even for $\beta < \frac12$, given that $\beta$ stays away from $0$, so it may be that our results generalise further.

The main novelty of this work is dealing with the two fractional differential operators in establishing the tools to prove the existence and uniqueness of solutions to the system \eqref{eq:weak_FPDE_0} in \cref{sec:existence}. 
The paper is structured as follows: \cref{sec:prelims} contains several self-contained subsections, establishing prerequisites for the main results. 
Most notably, \cref{thm:coercivity_bound} generalises a fundamental result in the analysis of elliptic equations to the case of the fractional Laplacian.
Furthermore, \cref{thm:cavalho_neto_fehlberg}, the main result from \cite{neto2020}, replaces a bound usually established using the Leibniz rule, which is unavailable for the Caputo fractional derivative.
\cref{sec:existence} and \cref{sec:regularity} show the existence, uniqueness and some regularity of a solution to \eqref{eq:weak_FPDE_0}, given some regularity assumptions on $V$, $f$ and $h$.
Theorems \ref{thm:approx_solns}, \ref{thm:energy_estimates}, \ref{thm:existence}, \ref{thm:uniqueness} and \ref{thm:regularity} adapt theorems found in Chapter 7 of \cite{evans2010}.
The regularity result in \cref{thm:regularity} is then used to show the $L^2$-convergence of numerically computable approximations to the solution, providing a numerical scheme that was also used to create the plots in this paper.

\section{Preliminaries}\label{sec:prelims}

Our main goal is to prove the existence of solutions to \eqref{eq:weak_FPDE_1}. 
Following a Galerkin approximation strategy, we will find approximate solutions in lower-dimensional subspaces. 
Then, we prove bounds on these approximations that we can use to pass to a limit to find a full solution.
This strategy requires many prerequisite results that are collected in this section.

\subsection{Fourier series}
We fix a dimension $d\in\Z_{>0}$ and define $\T := \R/(2\pi\Z)$ to be the one-dimensional periodic domain; recall that we consider the PDE on $\T^d$.
For $k\in\Z^d$, define $w_k : \T^d \to \C$ by
\begin{equation}\label{eq:basiselements}
    w_k(x) := \frac1{(2\pi)^{d/2}} \e^{i k \cdot x}.
\end{equation}
We note that $\{w_k\}_{k\in \Z^d}$ is an orthonormal basis of $L^2(\T^d, \C)$; for convenience, we denote this space as $L^2:=L^2(\T^d, \C)$.

Given a function $u \in L^2$, we define the Fourier coefficients of $u$ to be
\begin{equation*}
    \longft u (k) = \ft u (k) = \inprod{u}{w_k} = \int_{\T^d} u(x) \overline{w_k(x)} \d x = \frac1{(2\pi)^{d/2}}\int_{\T^d} u(x) \e^{-ik\cdot x} \d x,
\end{equation*}
for any $k \in \Z^d$, and then we have that
\begin{equation*}
    u(x) = \sum_{k\in \Z^d} \ft u(k) w_k(x),
\end{equation*}
with the equality in $L^2$ \cite[p22]{Rudin1962}.
Finally, the following result will allow us to move freely from equations in $L^2$ to equations in the Fourier space $l^2 := l^2(\Z^d,\C)$.

\begin{proposition}[Parseval's identity]\label{prop:parseval}
    For any function $u$, it holds that
    \begin{equation*}
        \norm[L^2]{u} = \norm[l^2]{\ft u}.
    \end{equation*}
    In particular, $u \in L^2$ if and only if $\ft u \in l^2$.
\end{proposition}

\subsection{Fractional Sobolev Spaces}
Fix $\beta > 0$. We focus on the case where $\beta \leq 1$, but this restriction is not necessary.
For our purposes, the following definition of the fractional Laplacian as a Fourier multiplier is most convenient \cite[p7]{lischke2020}:
\begin{definition}[Fractional Laplacian]\label{def:flap}
    If $\beta > 0$ and $u \in L^2$, then the Fractional Laplacian of $u$ of order $\beta$ is identified with the pseudodifferential operator defined by the Fourier multiplier $-|k|^{2\beta}$, that is,
    \begin{equation}\label{eq:flap}
        -\flap{\beta} u = \F^{-1}\left[-|k|^{2\beta}\F [u](k)\right]
    \end{equation}
    whenever the sum $\sum_{k\in\Z^d} |k|^{2\beta} \ft u (k)$ is finite.
\end{definition}

Equipped with this notion, we want to find the right space to look for solutions in. We start by defining the fractional Sobolev space
\begin{equation}
    H^\beta := H^\beta(\T^d, \C) := \left\{ u\in L^2(\T^d, \C) : \norm{\flap{\beta/2} u} < \infty \right\},
\end{equation}
equipped with the norm
\begin{equation}
    \norm[H^\beta]{u} := \left( \norm{\flap{\beta/2} u }^2 + \norm{u}^2 \right)^\frac12.
\end{equation}
In the case $\beta = 1$, this definition aligns with the usual definition of the space of periodic functions
\begin{equation}
    H^1 := \left\{ u\in L^2(\T^d,\C) : \norm{ |\nabla u| } < \infty \right\},
\end{equation}
since
\begin{equation}\label{eq:H1_norms_are_equivalent}
    \norm{\flap{1/2} u}^2 = \sum_{k\in \Z^d} |k|^2 |\ft u (k)|^2 = \norm{|\nabla u| }^2.
\end{equation}

For our analysis, it will be vital that we can bound lower-order derivatives by higher-order derivatives. The following result provides an appropriate version of Poincar\'e's inequality to do this.
\begin{proposition}[Poincar\'e's inequality] \label{prop:poincare} \
    \begin{enumerate}[(i)]
        \item Suppose $\beta > 0$; then for any $u \in H^\beta$ we have
        \begin{equation}\label{eq:poincare1}
            \norm{u - (u)} \leq  \norm{\flap{\beta/2} u},
        \end{equation}
        where
        \begin{equation*}
            (u) := \dashint_{\T^d} u(x,t) \d x
        \end{equation*}
        is the average of $u$ over $\T^d$.
        \item Suppose $0 < \gamma \leq \beta$. Then for any $u \in H^\beta$,
        \begin{equation*}
            \norm{\flap{\gamma/2} u} \leq \norm{\flap{\beta/2} u}.
        \end{equation*}
    \end{enumerate}
\end{proposition}

\begin{proof}
    Using \cref{prop:parseval} (Parseval's identity), we can express \eqref{eq:poincare1} into the following statement about Fourier coefficients:
    \begin{equation*}
        \norm{u-(u)}^2 = \sum_{k\in \Z^d \setminus \{0\}} |\ft u (k)|^2 \leq \sum_{k\in\Z^d} |k|^\beta |\ft u(k)|^2 = \norm{\flap{\beta/2} u}^2.
    \end{equation*}
    We note that subtracting the average $(u)$ in \eqref{eq:poincare1} yields the fact that the $k=0$ term in the sum is zero.
    The second inequality follows by a similar argument.
\end{proof}

The appearance of the average $(u)$ in \eqref{eq:poincare1} (and the fact from \cref{thm:constant_mass} that we will establish later that the mass of a function $u$ satisfying the weak PDE is constant in time) motivates working in a space of functions $u$ of constant mass $(u) = 0$. 
Thus, define $\sobsp\beta$ to be the closure of the the subspace
\begin{equation*}
    \sobsp\beta:=\overline{\{u \in C^\infty (\T^d) : (u) = 0 \}}^{H^\beta}
\end{equation*}
in the Hilbert space $H^\beta$ with its corresponding norm.

\begin{remark}\label{rmk:norms}
    The function $u \mapsto \norm{\flap{\beta/2} u}$ is a seminorm on $H^\beta$; on the other hand, due to the constraint that $(u) = 0$, it is a norm on $\sobsp\beta$ that we could use instead.
    However, there is no need since \cref{prop:poincare} (Poincar\'e's inequality) implies that on $\sobsp\beta$, the norms $\norm{\flap{\beta/2} \,\cdot\,}$ and $\norm[H^\beta]{\,\cdot\,}$ are equivalent.
\end{remark}

\subsection{The fractional time derivative}
Let us turn our attention to the Caputo derivative in the time variable.
Given a function $\varphi \in L^1 ([0,T], \R)$, we define its fractional integral $I^\alpha \varphi(t)$ of order $\alpha > 0$ to be \cite[p69]{kilbas2006}
\begin{equation}
    I^\alpha \varphi(t) := \frac1{\Gamma(\alpha)} \int_0^t (t-\tau)^{\alpha-1} \varphi(\tau) \d \tau.
\end{equation}

If instead $\varphi$ takes values in an arbitrary separable Banach space $X$ and we require $\varphi$ to be an element of the Bochner space $L^1([0,T],X)$, the fractional integral $I^\alpha\varphi (t)$ can be defined in the same way via Bochner integration.

Using this definition to fractionally integrate in time, we then differentiate to give the following definition for a fractional derivative.
\begin{definition}[Caputo fractional derivative] See \cite[p78ff]{podlubny1999}.
    Let $\alpha \in (0,1)$. Let $X$ be a separable Banach space and $\varphi \in C([0,T],X)$ such that $I^{1-\alpha}\varphi$ is absolutely continuous. 
    Then the Caputo fractional derivative of $\varphi$ is defined to be
    \begin{equation*}
        \deriv[\alpha]{\varphi}{t} := \deriv{}{t} \bigl(I^{1-\alpha} (\varphi(t) - \varphi(0)) \bigr).
    \end{equation*}
\end{definition}

\begin{remark}
    A corollary of Morrey's inequality is that the space of absolutely continuous functions $AC([0,T],X)$ is precisely the Sobolev space $W^{1,1}([0,T],X)$.
    Thus, requiring a function $\varphi$ to be absolutely continuous is the same as asking for its first weak derivative to exist \cite[p676]{neto2020}.
\end{remark}

Observe that in the above definition, we require $\varphi$ to be continuous only to make sense of $\varphi(0)$.
\emph{A priori}, we would like to be able to take the Caputo fractional derivative of $L^1$-functions.
Thus, in the following, we will establish a notion that allows us to make sense of $\varphi(0)$, motivated by the Lebesgue differentiation theorem; this is a special case of Proposition 2.10 in \cite[p93]{heinonen2001}.

\begin{theorem}[A version of the Lebesgue differentiation theorem] \label{thm:leb_diff_thm}
    Suppose $X$ is a separable Banach space and $\varphi\in L^1([0,T],X)$.
    Then we have that almost every $t \in [0,T]$ is a Lebesgue point of $\varphi$, that is, 
    \begin{equation}
        \lim_{\varepsilon \to 0} \dashint_{B_\varepsilon(t)} \norm[X]{\varphi(\tau) - \varphi(t)} \d \tau = 0.
    \end{equation}
\end{theorem}

Notice that the conclusion of \cref{thm:leb_diff_thm} holds for \emph{every} $t \in [0,T]$ if $\varphi$ is continuous as a consequence of the integral mean value theorem.
Furthermore, the theorem guarantees that reassigning values of $\varphi$ according to \eqref{eq:sense_ic} does not change the function in $L^1([0,T],X)$.
This means that it is reasonable to define

\begin{equation}\label{eq:sense_ic}
    \varphi(0) := \lim_{\varepsilon\to 0} \int_0^T \varphi(\tau) \eta_\varepsilon(\tau) \d \tau,
\end{equation}
where
\begin{equation}\label{eq:mollifier}
    \eta_\varepsilon(t) = \frac1\varepsilon \*1_{[0,\varepsilon)} (t).
\end{equation}
Choosing $\varphi(0)$ in this way corresponds to choosing the precise representative of the $L^1$-functions (see \cite[p138ff]{Rudin1987}).

We now have a way compatible with the remaining analysis to compute the Caputo fractional derivative of a function $\varphi\in L^1([0,T],X)$ that is not necessarily continuous at $0$, as long as $I^{1-\alpha} \varphi$ is absolutely continuous.

\subsection{Weak formulation}
Let us return to the equation \eqref{eq:weak_FPDE_0}. We seek a weak formulation of this equation. As such, we seek to pass part of the fractional Laplacian of $u$ onto a test function $v$.

\begin{lemma}[Divergence Theorem for the Fractional Laplacian]\label{lemma:flap_div_thm}
    For any $\beta, \beta' > 0$ such that we have that 
    \begin{equation*}
        \int_{\T^d} \flap{\beta} u\,\flap{\beta'} v \d x = \int_{\T^d} u \flap{\beta +\beta'} v \d x,    
    \end{equation*}
    whenever $u\in H^\beta$ and $v\in H^{\beta+\beta'}$.
\end{lemma}

\begin{proof}
    First, suppose that $u = w_k$ and $v = w_l$, where $l,k \in \Z^d$, where $w_k$ refers to the orthonormal Fourier basis elements introduced in \eqref{eq:basiselements}.
    Then
    \begin{equation*}
        \int_{\T^d} \flap{\beta} w_k \flap{\beta'} w_l \d x = |k|^\beta |l|^{\beta'} \inprod{w_k}{w_l}.
    \end{equation*}
    By construction, the latter inner product is zero unless $k = l$, and if $k=l$, then $|k|^\beta |l|^{\beta'}=|l|^{\beta+\beta'}$. As such,
    \begin{equation*}
        \int_{\T^d} \flap{\beta} w_k \flap{\beta'} w_l \d x = |l|^{\beta+\beta'} \inprod{w_k}{w_l} = \int_{\T^d} w_k \flap{\beta +\beta'} w_l \d x.
    \end{equation*}
    By linearity, the result holds for $u,v \in \mspan\left( \{w_k\}_k\right)$, but this space is dense in the corresponding Sobolev spaces $H^\beta$ and $H^{\beta+\beta'}$, so by continuity, we deduce the result.
\end{proof}

In light of the previous result, we define a bilinear form $B:H^\beta\times H^\beta\to \R$ as
\begin{equation*}
    B[u,v] := \int_{\T^d} K \flap{\beta/2} u \flap{\beta/2} v + (u\nabla V)\cdot \nabla v \d x.
\end{equation*}
With this definition, and applying the usual Divergence Theorem along with \cref{lemma:flap_div_thm}, we propose the following weak formulation of \eqref{eq:weak_FPDE_0}. We seek a function $u$ satisfying
\begin{equation}\begin{aligned}\label{eq:weak_FPDE_1}
    \Bigl\langle\pd[\alpha]{u}{t}, v \Bigr\rangle_{L^2} + B[u,v] &= \inprod{f}{v}  &&\text{ on } \T^d \text{ for almost every } t\in (0,T) \\
    u &= h  &&\text{ on } \T^d \text{ for } t = 0
\end{aligned}\end{equation}
for any $v \in \sobsp\beta$.

Let us obtain a version of \eqref{eq:weak_FPDE_1} in terms of the Fourier transform $\ft u$.
Observe that the Fourier coefficients of a product is the convolution of the Fourier coefficients, that is,
\begin{equation*}
    \longft{\varphi \phi} = \ft \varphi * \ft \phi,
\end{equation*}
where $*$ denotes the discrete convolution:
\[
\big(\hat{\varphi}*\hat{\phi}\big)(k) = \sum_{k'\in\Z^d}\hat{\varphi}(k-k')\hat{\phi}(k').
\]
Then for $v = w_k$, where $k \in \Z^d$, \eqref{eq:weak_FPDE_1} transforms into
\begin{equation}\label{eq:fpde_fourier_version}
    \pd[\alpha]{}{t}\ft u (k,t) + K |k|^\beta \ft u(k,t) - ik\cdot \left( \longft{\nabla V} * \ft u (\cdot, t)\right) = \ft f(k,t).
\end{equation}
Observe that condition \eqref{eq:source_term_condition} on the spatial integral of $f$ is equivalent to the statement that $\ft f(0,t) = 0$. 
In particular, if we take $k = 0$, \eqref{eq:fpde_fourier_version} trivially shows that $\pd[\alpha]{}{t} \ft u (0,t) = 0$, summarised in the following theorem:

\begin{theorem}[Conservation of mass]\label{thm:constant_mass}
    If $u$ is a weak solution to \eqref{eq:weak_FPDE_0}, where $\int_{\T^d} f(x,t) \d x = 0$ for almost every $t$, then the total mass $\int_{\T^d} u(x,t) \d x$ is constant in time.
\end{theorem}

Next, let us shift our perspective to the solution $u$. 
Define $\sobsp{-\beta} := \left(\sobsp\beta\right)^*$ to be the topological dual of $\sobsp\beta$. We note that 
\begin{itemize}
    \item $\sobsp\beta \subset L^2 \subset \sobsp{-\beta}$
    are all separable Hilbert spaces; and
    \item The fractional Sobolev space $\sobsp\beta$ is dense in $L^2$, and the embedding given by the identity map is continuous since $\norm{u} \leq \norm[H^\beta]{u}$
    for all $u \in \sobsp\beta$.
\end{itemize}
These observations mean that $\sobsp\beta \subset L^2 \subset \sobsp{-\beta}$ is a Gel'fand triple, which suggests that we look for solutions of the following form.
Define $\*u : [0,T] \to \sobsp\beta$ by
\begin{equation*}
    \*u(t) := u(\cdot, t).
\end{equation*}
Given that we expect some regularity in time, we expect $\*u$ to be an element of the Bochner space $L^2([0,T],\sobsp\beta)$.
Similarly, let us define $\*V, \*f  \in L^2([0,T],L^2)$ by
\begin{equation*}
    \*V(t) := V(\cdot, t)\quad
\text{and}\quad
    \*f(t) := f(\cdot, t).
\end{equation*}
We note that we will later require that $\*V$ and $\*f$ are elements of Bochner spaces with stronger regularity.

Now suppose that $w \in L^2 \subset \sobsp{-\beta}$ and $v \in \sobsp\beta$. Then we have
\begin{equation}\label{eq:pairing_inprod}
    \inprod[{\sobsp{-\beta},\sobsp\beta}]{w}{v} = \inprod{w}{v},
\end{equation}
(see for example equation (7.14) in \cite{roubicek2013})
where $\inprod[\sobsp{-\beta},\sobsp\beta]{w}{v}$ is the $\sobsp{-\beta}$--$\sobsp\beta$ pairing, applying $w$ to $v$.
From this point of view, let $\*u^{(\alpha)}$ be the Caputo fractional derivative of $\*u$ of order $\alpha$.
The weak formulation \eqref{eq:weak_FPDE_1} can then be read as
\begin{equation*}
    \inprod{\*u^{(\alpha)}(t)}{v} = \inprod{\*f(t)}{v} - B[\*u(t),v]\quad\text{for almost every }t\in(0,T).
\end{equation*}
The right-hand side can be treated as a linear functional acting on $v \in \sobsp\beta$, and this suggests that we can interpret the left-hand side as the pairing $\inprod[{\sobsp{-\beta},\sobsp\beta}]{\*u^{(\alpha)}}{v}$ applying $\*u^{(\alpha)}$ to $v \in \sobsp\beta$.
Thus, it is reasonable to look for a solution $\*u \in L^2([0,T], \sobsp\beta)$ whose Caputo fractional derivative of order $\alpha$ is $\*u^{(\alpha)} \in L^2([0,T],\sobsp{-\beta})$.

With these considerations in mind, \cref{sec:existence} and the remainder of \cref{sec:prelims} will establish the existence (and uniqueness) of a solution $\*u \in L^2([0,T], \sobsp\beta)$ with $\*u^{(\alpha)} \in L^2([0,T],\sobsp{-\beta})$ such that

\begin{equation}\begin{aligned}\label{eq:weak_FPDE}
    \inprod{\*u^{(\alpha)}}{v} + B[\*u,v] &= \inprod{f}{v}  && \text{ for almost every } t\in (0,T) \\
    \*u(0) &= h  && \text{ in the sense of \eqref{eq:sense_ic}}
\end{aligned}\end{equation}
for any $v \in \sobsp\beta$.

\subsection{A coercivity bound}

This subsection works towards establishing \emph{a priori} bounds on weak solutions $\*u$, which are encoded in \cref{thm:coercivity_bound}. The arguments are a generalisation of standard arguments for elliptic differential operators to the fractional case, but for completeness, we record the details in full.

\begin{lemma}[A Young-type inequality]\label{lemma:youngtype_ineq}
    Suppose $0 \leq \beta'' < \beta' < \beta$ and fix $\varepsilon > 0$ and sufficiently small. Then there exists a constant $c(\varepsilon)$ such that for any $u \in \sobsp{\beta}$,
    \begin{equation}
        \norm{\flap{\beta'} u}^2 \leq \varepsilon \norm{\flap{\beta} u}^2 + c(\varepsilon) \norm{\flap{\beta''} u}^2.
    \end{equation}
\end{lemma}

We will argue inductively, iteratively splitting the $\beta'$-term into a $(\beta'+\delta)$- and $(\beta'-\delta)$-term, until we hit $\beta''$. 
To keep the coefficient of the $\beta$-term arbitrarily small, we need the following Cauchy Young-type inequality \cite[p708]{evans2010}:
For $a,b \in \R$ and any $\varepsilon > 0$ we have that
\begin{equation} \label{eq:apx_Cauchy_Young_ineq}
    ab \leq \varepsilon a^2 + \frac{b^2}{4\varepsilon}.
\end{equation} 

\begin{proof}
    Define $\delta = \min\{\beta -\beta', \beta' -\beta'' \}$.
    We first use \cref{prop:parseval} to rewrite the expressions in Fourier space and then use \eqref{eq:apx_Cauchy_Young_ineq} to obtain the estimate
    \begin{equation}
        \begin{aligned}
            \norm{\flap{\beta'} u}^2 &= \sum_k |k|^{2(\beta'+\delta)}|\ft u(k)|   |k|^{2(\beta'-\delta)}|\ft u(k)| \\ 
        &\leq \sum_k \frac\varepsilon2 |k|^{4(\beta'+\delta)} |\ft u(k)|^2 +\frac1{2\varepsilon} |k|^{4(\beta'-\delta)}|\ft u(k)|^2  \\
        & = \frac\varepsilon2 \norm{\flap{\beta'+\delta} u}^2 + \frac1{2\varepsilon} \norm{\flap{\beta'-\delta} u}^2. 
        \end{aligned}
        \label{eq:proof_young_type}
    \end{equation}
    If $\delta = \beta'-\beta''$ so that the exponent is $\beta'-\delta=\beta''$ in the second term, we have $\beta'+\delta\leq\beta$, and we can employ \cref{prop:poincare} (Poincar\'e's inequality) to estimate the first term, i.e.
    \begin{equation*}
        \frac\varepsilon2 \norm{\flap{\beta'+\delta} u}^2 \leq \frac\varepsilon2 \norm{\flap{\beta} u}^2,
    \end{equation*}
    and we are done.
    
    Otherwise, we must have the opposite case, so that $\delta = \beta-\beta'$. We can repeat the same argument for the second term on the right-hand side of \eqref{eq:proof_young_type}, where we now replace $\beta'$ by $\beta' - \delta$ and $\varepsilon$ by $\frac{\frac\varepsilon2}{\frac1{2\varepsilon}} = \varepsilon^2$. Iterating this process further if necessary, we note that $\delta$ must increase linearly at each step until we reach the first case, and we can adjust the value of our initial $\varepsilon$ in order to ensure the coefficient in front of the resulting terms is controlled. We thereby arrive at the desired estimate.
\end{proof}

The following proposition is the key step in the proof of \cref{thm:coercivity_bound}. 
It also provides some control over all of the constants in the bound on the forcing term and allows for $\beta = \frac12$, which makes it a useful result by itself.

\begin{proposition}[Weak bound on the forcing term]\label{prop:weak_bound_forcing_term}
    Suppose the potential $V$ is regular enough to ensure that the constant
    \begin{equation}\label{eq:constant_from_potential}
        c(V) := \frac1{(2\pi)^{d/2}} \Big(\norm[l^1]{\longft{\nabla V}} + \norm[l^1]{\longft{\flap{3/4} V}}\Big)
    \end{equation}
    is finite, and suppose that $u, v \in \sobsp\beta$, where $\beta \geq \frac12$; then
    \begin{equation}\begin{aligned}\label{eq:weak_bound_1}
        (2\pi)^{d/2} \left|\int_{\T^d} \nabla V u \cdot \nabla v \d x \right| &\leq \norm[l^1]{\longft{\nabla V}} \norm{\flap{1/4} u} \norm{\flap{1/4} v} \\
        &\hspace{2cm}+ \norm[l^1]{\longft{\flap{3/4} V}} \norm{u} \norm{\flap{1/4} v},
    \end{aligned}\end{equation}
    and in particular,
    \begin{equation}\label{eq:weak_bound_2}
        \left|\int_{\T^d} \nabla V u \cdot \nabla v \d x \right| \leq c(V) \norm{\flap{1/4} u} \norm{\flap{1/4} v}.
    \end{equation}
\end{proposition}

\begin{proof}
    We will prove the inequality in Fourier space. Thus, we first need to transform the forcing term in terms of $V$, $u$ and $v$ into an expression involving $\ft V$, $\ft u$ and $\ft v$. Hence, we compute
    \begin{equation}
        \begin{aligned}
        \int_{\T^d} u\nabla V\cdot \nabla v \d x &= \sum_{j=1}^d \int_{\T^d} u\pd{V}{x_j}  \pd{v}{x_j} \d x \\
        &= \sum_{j=1}^d \int_{\T^d} \sum_{l\in \Z^d} \ft V(l) i l_j w_l \sum_{m\in\Z^d} \ft u(m) w_m \sum_{k\in\Z^d} \ft v(k) i k_j w_k \d x \\
        &= -\frac1{(2\pi)^{d/2}} \sum_{\substack{k,l,m\in\Z^d\\k+l+m = 0}}  \ft V(l) \ft u(m) \ft v(k) l \cdot k,
        \end{aligned}\label{eq:forcing_term0}
    \end{equation}
    where $\cdot$ is the standard vector dot product. 
    
    Let us start by noting that
    \begin{equation*}
        \sum_{k\in\Z^d}\left|\ft u(-l-k)\right|^2 |k| = \sum_{m\in\Z^d} \left| \ft u(m)\right|^2 |-m-l| \leq \sum_{m\in\Z^d} \left|\ft u (m) \right|^2 |m| + |l| \sum_{m\in\Z^d} \left| \ft u (m) \right|^2.
    \end{equation*}
    Thus, applying the Cauchy-Schwarz inequality and the inequality just derived, we obtain
    \begin{equation}
    \begin{aligned}
        &\left|\sum_{k\in\Z^d} \ft u(-l-k) \ft v(k) k \right| \\
        &\qquad\leq \sum_{k\in\Z^d} \left|\ft u(-l-k) |k|^{\frac12} \right| \left| \ft v(k) |k|^{\frac12} \right| \\
        &\qquad\leq \bigg(\sum_{k\in\Z^d} \left|\ft u(-l-k)\right|^2 |k|\bigg)^\frac12\bigg( \sum_{k\in\Z^d}\left|\ft v(k)\right|^2 |k| \bigg)^\frac12\\
        &\qquad\leq \bigg(\sum_{k\in\Z^d} \left|\ft u (k) \right|^2 |k| + |l| \sum_{k\in\Z^d} \left| \ft u (k) \right|^2\bigg)^\frac12
        \bigg(\sum_{k\in\Z^d} \left|\ft v(k)\right|^2 |k|\bigg)^\frac12 \\
        &\qquad\leq \bigg(\sum_{k\in\Z^d} \left|\ft u(k)\right|^2 |k|\bigg)^\frac12 \bigg(\sum_{k\in\Z^d} \left|\ft v(k)\right|^2 |k|\bigg)^\frac12 + |l|^\frac12 \bigg(\sum_{k\in\Z^d} \left| \ft u (k) \right|^2\bigg)^\frac12 \bigg( \sum_{k\in\Z^d} \left|\ft v (k) \right|^2 |k|\bigg)^\frac12 \\
        &\qquad= \norm{\flap{1/4} u} \norm{\flap{1/4} v} + |l|^\frac12 \norm{u} \norm{\flap{1/4} v}.
    \end{aligned} \label{eq:forcing_term1}
    \end{equation}
    Returning once more to the identity obtained in \eqref{eq:forcing_term0} and using the estimate in \eqref{eq:forcing_term1}, we can bound
    \begin{equation}
        \begin{aligned}
        \left|\sum_{k+l+m = 0} \ft V(l) \ft u(m) \ft v(k) l \cdot k \right| \hspace{-4cm}&\hspace{4cm}\leq \sum_{l\in\Z^d} \left| \ft V(l) l \right| \left|\sum_{k\in\Z^d} \ft u(-l-k) \ft v(k) k \right| \\
        &\leq \sum_{l\in\Z^d} \left| \ft V(l)\right|\left| l \right| \norm{\flap{1/4} u} \norm{\flap{1/4} v} + \sum_{l\in\Z^d} \left| \ft V(l) \right| |l|^\frac32 \norm{u} \norm{\flap{1/4} v} \\
        &= \norm[l^1]{\longft{\nabla V}} \norm{\flap{1/4} u} \norm{\flap{1/4} v} + \norm[l^1]{\longft{\flap{3/4} V}} \norm{u} \norm{\flap{1/4} v}. 
    \end{aligned}\label{eq:weak_bound_proof}
    \end{equation}
    This completes the proof of \eqref{eq:weak_bound_1}.
    To yield \eqref{eq:weak_bound_2}, we now apply \cref{prop:poincare} to $\norm{u}$ in \eqref{eq:weak_bound_proof}.
\end{proof}

\begin{remark}
    The constant $c(V)$ arising from the potential $V$ in the inequalities in \cref{prop:weak_bound_forcing_term} are given in terms of the $l^1$-norm of the Fourier transform of derivatives of $V$.
    It would be convenient to compare the $l^1$-norm of a function's Fourier transform directly to the function's $L^\infty$-norm, in other words, to establish an equality up to constants between $\norm[L^\infty]{f}$ and $\| \ft f \|_{l^1}$.
    The triangle inequality immediately gives that for $f\in L^\infty$,
    \begin{equation} \label{eq:special_hausdorff-young}
        \norm[L^\infty]{f} \leq \| \ft f \|_{l^1}.
    \end{equation}
    However, the reverse inequality is not true, as noted in the discussion about the Hausdorff-Young Theorem, which gives a more general version of \eqref{eq:special_hausdorff-young}, in \cite[p124]{katznelson2004}.
    An explicit counterexample is the continuous $2\pi$-periodic function $f(x) = \sum_{n=2}^\infty \frac{\e^{in\log n}}{n^{1/2} (\log n)^2} \e^{inx}$, which satisfies
    $\sum_n |\ft f (n)| = \infty$ \cite[p199f]{zygmund2003}.
\end{remark}

\begin{remark}\label{rmk:suff_cond_on_V}
    If $V \in H^3$, then the constant $c(V)$ exists.
    Indeed, then by \cref{prop:parseval}, 
    \begin{equation*}
        \sum_{l\in \Z^d} \left(\left|\ft V (l) \right| |l|^3 \right)^2 = \norm[l^2]{\longft{D^3 V}}^2 = \norm{D^3 V}^2 < \infty.
    \end{equation*}
    In particular,
    \begin{equation*}
        \left|\ft V(l)\right| |l|^3 \leq C
    \end{equation*}
    is bounded uniformly in $l$ by some constant $C$ depending on $V$.
    Hence, 
    \begin{equation}
        \norm[l^1]{\longft{\flap{3/4} V}} = \sum_{l\in\Z^d} |\ft V(l)| |l|^{\frac32} \leq C \sum_{l \in \Z^d} \frac1{|l|^\frac32} < \infty,
    \end{equation}
    and similarly for $\norm[l^1]{\longft{\nabla V}}$.
    Observe that this condition can be weakened; we merely aim to provide a sufficient and easier-to-check requirement on $V$.
\end{remark}

\begin{proposition}[Strong bound on the forcing term]
    Suppose $V$ is sufficiently regular to guarantee that $c(V)$ defined in \eqref{eq:constant_from_potential} is finite, $\beta > \frac12$ and $u \in \sobsp\beta$. 
    Then, for any $\varepsilon > 0$, we have that
    \begin{equation}\label{eq:strong_bound}
        \left|\int_{\T^d} \nabla V u \cdot \nabla u \d x \right| \leq \varepsilon \norm{\flap{\beta/2} u}^2 + c(V) c(\varepsilon) \norm{u}^2,
    \end{equation}
    where the constant $c(\varepsilon)$ is identical to that in  the statement of \cref{prop:weak_bound_forcing_term}.
\end{proposition}

\begin{proof}
    Substitute $v = u$ into \eqref{eq:weak_bound_2} and employ \cref{lemma:youngtype_ineq} to conclude
    \begin{equation*}
        \left|\int_{\T^d} \nabla V u \cdot \nabla u \d x \right| \leq c(V) \norm{\flap{1/2} u}^2 \leq \varepsilon \norm{\flap{\beta/2} u}^2 + c(V) c(\varepsilon) \norm{u}^2.\qedhere
    \end{equation*}
\end{proof}

Finally, we can use \eqref{eq:strong_bound} to show the following key theorem for the bilinear form $B[\cdot,\cdot]$; compare with \cite[p320]{evans2010}.

\begin{theorem}[Coercivity bound]\label{thm:coercivity_bound}
    Suppose that $K > 0$, $c(V) \leq M_1$ for some $M_1$, $\beta > \frac12$. 
    \begin{enumerate}[(i)]
        \item There exist constants $\gamma_1 > 0$, $\gamma_2 \geq 0$ depending on $K$, $M_1$ and $\beta$ such that
    \begin{equation}\label{eq:coercivity_bound}
        \gamma_1 \norm[H^\beta]{u}^2 \leq B[u,u] + \gamma_2 \norm{u}^2
    \end{equation}
    for any $u\in\sobsp\beta$.
    \item If we make the stronger hypothesis $K > c(V)$ and $\beta \geq \frac12$, there is just one constant $\gamma > 0$ such that 
    \begin{equation}\label{eq:weak_coercivity_bound}
        \gamma \norm[H^\beta]{u}^2 \leq B[u,u] + \gamma \norm{u}^2
    \end{equation}
    for any $u\in\sobsp\beta$.
    \end{enumerate}
\end{theorem}

\begin{proof}
    By definition,
    \begin{equation*}
        K \int_{\T^d} \left| \flap{\beta/2} u \right|^2 \d x = B[u,u] + \int_{\T^d} u\nabla V\cdot\nabla u \d x.
    \end{equation*}
    Choosing $0 <\varepsilon < K$ in \eqref{eq:strong_bound}, we estimate the second term on the right-hand side above:
    \begin{align*}
        K \norm{\flap{\beta/2} u}^2 &\leq B[u,u] + \varepsilon \norm{\flap{\beta/2} u}^2 + c(V)c(\varepsilon) \norm{u}^2.
    \end{align*}
    Rearranging gives
    \begin{equation}
        \left(K-\varepsilon\right) \left(\norm{u}^2 + \norm{\flap{\beta/2} u}^2\right) \leq B[u,u] + \left(K-\varepsilon + M_1c(\varepsilon)\right) \norm{u}^2,
    \end{equation}
    proving \eqref{eq:coercivity_bound}.
    To show \eqref{eq:weak_coercivity_bound}, we may take $\varepsilon = c(V)$ and use the same strategy with the weak bound on the forcing term \eqref{eq:weak_bound_2}.
\end{proof}

\subsection{Initial results}

Later, our strategy will be to start by finding a sequence $\*u_m$ of approximate solutions on finite-dimensional subspaces of $\sobsp\beta$. 
In order to deal with the Caputo fractional derivative in time, we need the following fractional version of Picard-Lindel\"of's Theorem, which we adapt from Theorem~4.1 and Remark~3.1 in \cite{sin2016}.

\begin{theorem}[Existence and Uniqueness for Fractional Linear ODEs]\label{thm:ode_existence_and_uniqueness}
    If $\alpha \in (0,1]$ and $g : [0,T] \times \C^n \to \C^n$ is a continuous function which is globally Lipschitz continuous in the $\C^n$-variable, then the system 
    \begin{equation}\begin{aligned}
        \deriv[\alpha]{}{t} y (t) &= g(t,y(t)), \\
        y(0) &= y_0
    \end{aligned}\end{equation}
    has a unique global solution $y \in C([0,T], \C^n)$.
\end{theorem}

We will then need to prove some bounds for these approximate solutions $\*u_m$ before passing to limits.
Gr\"onwall's inequality is an incredibly powerful tool to find such bounds; however, since we work with a Caputo fractional derivative, we need an adapted version.

\begin{theorem}[differential version of Gr\"onwall's inequality for the Caputo fractional derivative]\label{thm:gronwall}
    Suppose that $\alpha \in (0,1]$ and $\varphi \in L^1([0,T],\R)$ is a nonnegative function such that $I^{1-\alpha}\varphi$ is absolutely continuous.
    Let $c > 0$ and suppose $\zeta : [0,T] \to \R$ is a nonnegative nondecreasing integrable function. If
    \begin{equation}\label{eq:gronwall_hypothesis}
        \deriv[\alpha]{}{t} \varphi(t) \leq c\, \varphi(t) + \zeta(t) 
    \end{equation}
    for almost every $t \in [0,T]$, then
    \begin{equation}\label{eq:gronwall}
        \varphi(t) \leq (I^\alpha \zeta(t) + \varphi(0)) E_\alpha(ct^\alpha),
    \end{equation}
    where $E_\alpha$ is the Mittag-Leffler function defined via the power series
    \begin{equation}
        E_\alpha(z):=\sum_{k=0}^\infty \frac{z^k}{\Gamma(\alpha k+1)},
    \end{equation}
    and $\Gamma(z) := \int_0^\infty t^{z-1}\e^{-t}\d t$ is the usual Gamma function.
\end{theorem}
\begin{proof}
    This is a Corollary of Theorem 8 in \cite[p9f]{almeida2017}.
    From Lemma 2.5 (b) in \cite[p74f]{kilbas2006} (also compare Lemma 2.22 in \cite[p96]{kilbas2006}), we conclude that since $I^{1-\alpha}\varphi$ is absolutely continuous by assumption,
    \begin{equation}
        I^\alpha \left( \deriv[\alpha]{}{t} \varphi(t) \right) = \varphi(t) - \varphi(0).
    \end{equation}
    $I^\alpha$ is a monotone linear operator when acting on non-negative functions, so applying it to both sides of \eqref{eq:gronwall_hypothesis} preserves the inequality, yielding
    \begin{equation}
        \varphi(t) - \varphi(0) \leq c I^\alpha \varphi(t) + I^\alpha \zeta(t).
    \end{equation}
    Since $\zeta$ is assumed to be both nonnegative and nondecreasing, $I^\alpha \zeta$ inherits both properties, and we may apply Theorem 8 in \cite[p9f]{almeida2017} to deduce \eqref{eq:gronwall}.
\end{proof}

Once approximate solutions $\*u_m$ in an $m$--dimensional subspace are available, we want to send $m$ to infinity.
The following three results relate to the current discussion and enable us to pass properties of $\*u_m$ to the limit $\*u$.

\begin{lemma}[Preservation of initial condition]\label{lemma:limit_ic}
    Suppose that 
    \begin{equation}
        \*u_m \xrightharpoonup[]{} \*u \qquad \text{ weakly in } L^2\left([0,T],\sobsp\beta\right),
    \end{equation}
    and $\inprod{\*u_m(0)}w = \inprod{h}w$ for every $w \in \mspan\left( \{w_k\}_{|k|\leq m}\right)$, in the sense of \eqref{eq:sense_ic}. Then 
    \begin{equation}
        \*u(0) = h
    \end{equation}
    again in the sense of \eqref{eq:sense_ic}.
\end{lemma}

\begin{proof}
    Given $\varepsilon > 0$, recall the definition of $\eta_\varepsilon$ from \eqref{eq:mollifier}.
    Fix an arbitrary $w \in \mspan \left(\{w_k\}_{k\in \Z^d}\right)$.
    Then, passing to weak limits,
    \begin{equation}
        \int_0^T \inprod{\*u_m(t)}{\eta_\varepsilon (t) w} \d t \to \int_0^T \inprod{\*u(t)}{\eta_\varepsilon (t) w} \d t \qquad \text{ as } m\to \infty.
    \end{equation}
    Moreover, using \eqref{eq:sense_ic} for fixed $m$, we have
    \begin{equation}
        \int_0^T \inprod{\*u_m(t)}{\eta_\varepsilon (t) w} \d t \to \inprod{\*u_m(0)}w = \inprod{h}w\qquad \text{ as } \varepsilon\to 0,
    \end{equation}
    where the latter equality holds by assumption.
    Applying the triangle inequality, we may therefore conclude that
    \begin{equation*}
        \int_0^T \inprod{\*u(t)}{\eta_\varepsilon (t) w} \d t \to \inprod{h}w \qquad \text{ as } \varepsilon \to 0,
    \end{equation*}
    and thus by \eqref{eq:sense_ic},
    \begin{equation*}
        \inprod{\*u(0)}w = \inprod{h}w.
    \end{equation*}
    Since $w \in \mspan \left(\{w_k\}_{k\in \Z^d} \right)$ is arbitrary and the span (excluding the $k=0$ term) is dense in $\sobsp\beta$, the conclusion follows.
\end{proof}

Since we will deal with weak limits at first, we first need to find the right linear functional to consider, given by the following Lemma:
\begin{lemma}[Bounded linear functional]\label{lemma:bdd_lin_fnl}
    For any fixed $\phi \in C_c^\infty(0,T)$ and $w \in \sobsp\beta$, we have that $\Phi: L^2([0,T],\sobsp\beta) \to \C$ defined by
    \begin{equation}
        \Phi(\*v) := \int_0^T \inprod{I^{1-\alpha} \*v (t)}{\phi'(t) w} \d t
    \end{equation}
    is a bounded linear functional.
\end{lemma}
\begin{proof}
    Clearly, the functional is linear if it is well-defined.
    Now, using the fact that $I^{1-\alpha} : L^2 \to L^2$ is a bounded linear operator from Theorem 2.6 in \cite[p48]{samko1992}, the continuity of the embedding $\sobsp\beta \to L^2$ and the Cauchy-Schwarz inequality twice, we know that there is a constant $C$ such that 
    \begin{align*}
        \Phi(\*v) &\leq \int_0^T \left|\inprod{I^{1-\alpha} \*v (t)}{\phi'(t) w}\right| \d t \\
        &\leq \int_0^T \norm{I^{1-\alpha} \*v (t)}|\phi'(t)| \norm{w} \d t \\
        &\leq \norm{w}\int_0^T C\norm{\*v (t)}|\phi'(t)| \d t \\
        &\leq \norm{w} C \norm[{L^2([0,T])}]{\phi'} \norm[{L^2([0,T],\sobsp\beta)}]{\*v},
    \end{align*}
    so the functional is bounded.
\end{proof}

Finally, we need to be able to pass to limits while maintaining the relationship between the solution and its Caputo fractional derivative. 
\begin{proposition}[Bochner limits]\label{prop:vprime_equals_u}
    Compare to Problem 5 of Chapter 7 in \cite[p449]{evans2010}.
    Assume that
    \begin{align*}
        \*u_m &\xrightharpoonup[]{} \*u \qquad \text{ weakly in } L^2\left([0,T],\sobsp\beta\right) \\
        \*u_m^{(\alpha)} &\xrightharpoonup[]{} \*v \qquad \text{ weakly in } L^2\left([0,T],\sobsp{-\beta}\right).
    \end{align*}
    Then $\*v = \*u^{(\alpha)}$.
\end{proposition}

\begin{proof}
    Let $\phi \in C_c^\infty(0,T)$ and $w \in \sobsp\beta$ be arbitrary.
    Then we have, by definition, for every $m$ that
    \begin{equation}
        \int_0^T \inprod[{\sobsp{-\beta},\sobsp\beta}]{\*u_m^{(\alpha)}(t)}{\phi(t)w} \d t = - \int_0^T \inprod{I^{1-\alpha}\big(\*u_m(t) - \*u_m(0)\big)}{\phi'(t) w} \d t.
    \end{equation}
    Using \cref{lemma:bdd_lin_fnl} on the right-hand side, we may pass to weak limits on both sides to obtain
    \begin{equation}
        \int_0^T \inprod[{\sobsp{-\beta},\sobsp\beta}]{\*v(t)}{\phi(t)w} \d t = - \int_0^T \inprod{I^{1-\alpha}\big(\*u(t) - \*u(0)\big)}{\phi'(t) w} \d t.
    \end{equation}
    Since this equation holds for every $w \in \sobsp\beta$, we conclude 
    \begin{equation}
        \int_0^T \phi(t) \*v(t) \d t = \int_0^T \phi(t) \*u^{(\alpha)}(t) \d t,
    \end{equation}
    for every $\phi \in C_c^\infty(0,T)$, i.e. $\*v = \*u^{(\alpha)}$.
\end{proof}

Before we show the existence of solutions, we need one final important ingredient to replace the Leibnitz rule for integer-order differentiation. 
Instead of the equality $\deriv{}{t} \norm{\*u(t)}^2 = 2\inprod[{\sobsp{-\beta},\sobsp\beta}]{\deriv{\*u}{t}}{\*u(t)}$, we will use the following fractional version, which is taken from Theorem~4.16 part (b) in \cite{neto2020}.

\begin{theorem}\label{thm:cavalho_neto_fehlberg}
    Suppose that $\*u \in L^2([0,T],\sobsp\beta)$, $\*u^{(\alpha)} \in L^2([0,T],\sobsp{-\beta})$ and that $I^{1-\alpha} \norm{\*u(t)}^2$ is absolutely continuous.
    Then after a possible reidentification, we may assume that $\*u \in C([0,T],L^2)$ and
    \begin{equation}
        \deriv[\alpha]{}{t} \norm{\*u (t)}^2 \leq 2 \inprod[{\sobsp{-\beta},\sobsp\beta}]{\*u^{(\alpha)}(t)}{\*u(t)}.
    \end{equation}
\end{theorem}

\begin{remark}
    The issue of defining an initial condition precisely for the purposes of ensuring the fractional Caputo derivative is well-defined is only briefly discussed in \cite{neto2020}.
    However, we observe that it is \emph{a posteriori} demonstrated that $u$ is continuous by showing that a reasonable sequence of mollifications of $u$ is Cauchy in $C([0,T],L^2)$.
    This means that our choice to define $u(0)$ in the manner given in \eqref{eq:sense_ic} does not violate any steps taken in \cite{neto2020}.
\end{remark}

\section{Existence and uniqueness of weak solutions} \label{sec:existence}

The following two subsections, concerning existence and uniqueness, adapt strategies outlined in \cite[p377ff]{evans2010} to our fractional differential equation.
In particular, Theorems \ref{thm:approx_solns}, \ref{thm:energy_estimates}, \ref{thm:existence} and \ref{thm:uniqueness} correspond to Theorems 1, 2, 3 and 4 on pages 378--381 of \cite{evans2010}, respectively.

\subsection{Existence results}

\begin{theorem}[Approximate solutions]\label{thm:approx_solns}
    Suppose that all coefficients of the equation are continuous in time, i.e. $\*V \in C([0,T],H^1)$ and $\*f \in C([0,T],L^2)$.
    Then, for every positive integer $m$ there is a unique function $\*u_m:\T^d\times [0,T] \to \R$ of the form
    \begin{equation}\label{eq:approx_soln_form}
        \*u_m(t) = \sum_{|k|\leq m} d_m^k(t) w_k,
    \end{equation}
    where the coefficients $d_m^k(t)$ satisfy
    \begin{equation}\label{eq:approx_eqn_ic}
        d_m^k(0) = \inprod{h}{w_k},
    \end{equation}
    solving the weak equation
    \begin{equation}\label{eq:approx_eqn}
        \inprod[{\sobsp{-\beta},\sobsp\beta}]{\*u_m^{(\alpha)}}{w_k} + B[\*u_m,w_k] = \inprod{\*f}{w_k} 
    \end{equation}
    for all times $0 \leq t \leq T$ and $|k| \leq m$.
\end{theorem}

\begin{proof}
    Insert \eqref{eq:approx_soln_form} into the equation given in \eqref{eq:approx_eqn}. 
    Using the fact from \eqref{eq:pairing_inprod} that the $\sobsp{-\beta}$--$\sobsp\beta$ pairing is equivalent to the $L^2$ inner product whenever both make sense, we obtain the system of equations in $|k| \leq m$,
    \begin{equation*}
        \deriv[\alpha]{}{t} d_m^k(t) + \sum_{|j| \leq m} B[w_j,w_k] d_m^j(t) = \*{\ft{f}} (k,t).
    \end{equation*}
    If we define $g: [0,T] \times \C^{(2m+1)^d} \to \C^{(2m+1)^d}$, the spaces indexed by $|k| \leq m$, by
    \begin{equation*}
        g_k(t,y) = \*{\ft{f}}(k,t) - \sum_{|j| \leq m} B[w_j,w_k;t] y_j,     
    \end{equation*}
    then the equation reads
    \begin{equation}\label{eq:approx_soln_proof}
        \deriv[\alpha]{}{t} d_m^k(t) = g_k\left( t,\bigl(d_m^j(t)\bigr)_{|j|\leq m}\right).
    \end{equation}
    Since
    \begin{equation*}
        \big|\ft{\*f}(k,t) - \ft{\*f}(k,\tau) \big| = \left| \inprod{{\*f}(\cdot,t)}{w_k} - \inprod{{\*f}(\cdot,\tau)}{w_k} \right| \leq \norm{{\*f}(\cdot,t) - {\*f}(\cdot,\tau)}
    \end{equation*}
    for arbitrary $t$ and $\tau$ and fixed $k$, the assumptions on $\*V$ and $\*f$ imply that $g$ is continuous on $[0,T]\times \C^n$ and Lipschitz in $y$. Then \cref{thm:ode_existence_and_uniqueness} yields that subject to the initial conditions \eqref{eq:approx_eqn_ic}, the equation \eqref{eq:approx_soln_proof} has a unique continuous solution.
\end{proof}

Given the approximate solutions $\*u_m$, we want to pass to a limit.
The next theorem provides the bounds needed and uses many of the results from \cref{sec:prelims}.
First, suppose that there exists
\begin{equation}
    M_1 := \sup_{t\in[0,T]}c(\*V(t)),
\end{equation}
where $c(\cdot)$ is the constant defined in \eqref{eq:constant_from_potential}.
\begin{theorem}[Energy estimates]\label{thm:energy_estimates}
    Suppose $\beta > \frac12$. Assume that the constant in \eqref{eq:constant_from_potential} exists (see \cref{rmk:suff_cond_on_V}). Then there exists a constant $C$, depending only on $\alpha$, $\beta$, the final time $T$, the diffusion constant $K$ and the bound $M_1$, such that
    \begin{equation}\label{eq:energy_estimates}
        \sup_{0\leq t \leq T} \norm{\*u_m(t)} + \norm[{L^2\left([0,T],\sobsp\beta\right)}]{\*u_m} + \norm[{L^2\left([0,T],\sobsp[]{-\beta}\right)}] {\*u_m^{(\alpha)}} \leq C \left(\norm{h} + \norm[{L^\infty([0,T],L^2)}]{\*f}\right).
    \end{equation}
\end{theorem}

\begin{proof}
    For all $0 \leq t \leq T$, we multiply equation \eqref{eq:approx_eqn} by $d_m^k(t)$ and sum over $|k| \leq m$ to find the equation
    \begin{equation}\label{eq:sum_approx_eqn}
        \inprod[{\sobsp{-\beta},\sobsp\beta}]{\*u_m^{(\alpha)}}{\*u_m} + B[\*u_m,\*u_m] = \inprod{\*f}{\*u_m}.
    \end{equation}
    Combining the Cauchy-Schwarz and Young's inequality, we have $|\inprod{\*f}{\*u_m}| \leq \norm{\*f}\norm{\*u_m} \leq \frac12 (\norm{\*f}^2 + \norm{\*u_m}^2)$.
    Recalling \cref{thm:cavalho_neto_fehlberg}, we have that $\deriv[\alpha]{}{t}\left(\norm{\*u_m}^2\right) \leq 2\inprod[{\sobsp{-\beta},\sobsp\beta}]{\*u_m^{(\alpha)}}{\*u_m}$. Thus, we can use \eqref{eq:sum_approx_eqn} and \eqref{eq:coercivity_bound} to estimate 
    \begin{align}
        \deriv[\alpha]{}{t}\left(\norm{\*u_m}^2\right) + 2\gamma_1 \norm[H^\beta]{\*u_m}^2 &\leq 2 \left( \gamma_1 \norm[H^\beta]{\*u_m}^2 + \inprod{\*f}{\*u_m} - B[\*u_m,\*u_m] \right) \\
        &\leq (2\gamma_2 +1) \norm{\*u_m}^2 + \norm{\*f}^2.\label{eq:energy_estimates_proof}
    \end{align}
    Note that the dependency of $\gamma_1$ and $\gamma_2$ on $c(V)$ in \cref{thm:coercivity_bound} means that we can choose both constants using the uniform bound $M_1$ instead of $c(V)$.
    
    In particular, we have established $\deriv[\alpha]{}{t}\left(\norm{\*u_m}^2\right) \leq (2\gamma_2+1) \norm{\*u_m}^2 + \norm{\*f}^2$, so the differential version of Gr\"onwall's Lemma in \cref{thm:gronwall} yields
    \begin{equation}
        \norm{\*u_m(t)}^2 \leq \left(I^{\alpha} \norm{\*f}^2 + \norm{\*u_m(0)}^2\right) E_\alpha\bigl((2\gamma_2 +1) t^\alpha\bigr) \leq \norm{h}^2 E_\alpha\bigl((2\gamma_2 + 1) T^\alpha\bigr)
    \end{equation}
    and thus, setting $c = E_\alpha\bigl((2\gamma_2 + 1) T^\alpha\bigr) \max\{1, \frac1{\Gamma(\alpha+1)} T^\alpha \}$,
    \begin{equation}\label{eq:energy_estimate_first_term}
        \sup_{0\leq t\leq T}\norm{\*u_m(t)}^2 \leq c \left(\norm{h}^2 + \norm[{L^\infty([0,T],L^2)}]{\*f}^2\right).
    \end{equation}

    To control the second term on the left-hand side of \eqref{eq:energy_estimates}, apply $I^\alpha$ to both sides of \eqref{eq:energy_estimates_proof} to reach
    \begin{equation}
        \norm{\*u_m(t)}^2 - \norm{h}^2 + 2 \gamma_1 I^\alpha \norm[H^\beta]{\*u_m(t)}^2 \leq (2\gamma_2 + 1) I^\alpha \norm{\*u_m(t)}^2 + I^\alpha\norm{\*f}^2.
    \end{equation}
    Dropping the first term, rearranging the second term, applying \cref{prop:integral_poincare} to the third term and using \eqref{eq:energy_estimate_first_term} for the right-hand side gives for $t = T$,
    \begin{equation}
        \frac{2\gamma_1}{\Gamma(\alpha) T^{1-\alpha}} \int_0^T \norm[H^\beta]{\*u_m(t)}^2 \d t \leq \max \left\{1,\frac{T^\alpha}{\Gamma(\alpha+1)}\right\} (2\gamma_2 + 2) c \left(\norm{h}^2 + \norm[{L^\infty([0,T],L^2)}]{\*f}^2\right).
    \end{equation}
    Consequently, for an appropriate $c' > 0$, we conclude that 
    \begin{equation}
        \norm[{L^2\left([0,T],\sobsp\beta\right)}]{\*u_m(t)}^2 \leq c' \left(\norm{h}^2 + \norm[{L^\infty([0,T],L^2)}]{\*f}^2 \right).
    \end{equation}

    Moving on to the third term on the left-hand side of \eqref{eq:energy_estimates}, consider an arbitrary $v \in \sobsp\beta$ such that $\norm[H^\beta]{v} \leq 1$.
    Since $\{w_k\}_{|k|\leq m}$ is an orthonormal sequence in $L^2$, we can find a unique orthogonal decomposition of the form $v = v_1 + v_2$, satisfying $v_1 \in \mspan \left(\{w_k\}_{|k|\leq m}\right)$ and $v_2 \in L^2$ with $\inprod[L^2]{v_2}{w_k} = 0$ for all $|k| \leq m$.
    Furthermore, since $\mspan \left(\{w_k\}_{|k|\leq m}\right) \subset \sobsp \beta$ and  because $\{w_k\}_{|k|\leq m}$ is orthogonal in $\sobsp\beta$, we have that $\norm[H^\beta]{v_1} \leq \norm[H^\beta]v \leq 1$.
    From \eqref{eq:approx_eqn} we deduce that
    \begin{equation}
        \inprod[{\sobsp{-\beta},\sobsp\beta}]{\*u_m^{(\alpha)}}{v_1} + B[\*u_m,v_1] = \inprod{\*f}{v_1}.
    \end{equation}    
    Since $\*u_m^{(\alpha)} \in L^2$, we can apply \eqref{eq:pairing_inprod} to establish
    \begin{equation}
        \inprod[{\sobsp{-\beta},\sobsp\beta}]{\*u_m^{(\alpha)}}{v} = \inprod{\*u_m^{(\alpha)}}{v} = \inprod{\*u_m^{(\alpha)}}{v_1} = \inprod[{\sobsp{-\beta},\sobsp\beta}]{\*u_m^{(\alpha)}}{v_1} = \inprod{\*f}{v_1} - B[\*u_m,v_1],
    \end{equation}
    which implies that
    \begin{align}
        \left| \inprod[{\sobsp{-\beta},\sobsp\beta}]{\*u_m^{(\alpha)}}{v_1}\right| &\leq K \left|\inprod{\flap{\beta/2} \*u_m}{\flap{\beta/2} v_1} \right| + \left| \int_{\T^d} \nabla \*V u \cdot \nabla v_1 \d x \right| + \left|\inprod{\*f}{v_1}\right| \\
        &\leq (K+M_1) \norm{\flap{\beta/2} \*u_m} \norm{\flap{\beta/2} v_1} + \norm{\*f} \norm{v_1} \\
        &\leq \left((K+M_1) \norm[H^\beta]{\*u_m} + \norm{\*f}\right) \norm[H^\beta]{v},
    \end{align}
    using the Cauchy-Schwarz inequality, Poincar\'e's inequality repeatedly and \cref{prop:weak_bound_forcing_term}.
    Thus, 
    \begin{equation}
        \norm[{\sobsp[]{-\beta}}]{\*u_m^{(\alpha)}} \leq (K+M_1) \norm[H^\beta]{\*u_m} + \norm{\*f}.
    \end{equation}
    Squaring and integrating this inequality from $0$ to $T$ yields the desired estimate
    \begin{align}
        \norm[{L^2\left([0,T],\sobsp[]{-\beta}\right)}]{\*u_m^{(\alpha)}}^2 = \int_0^T \norm[{\sobsp[]{-\beta}}]{\*u_m^{(\alpha)}}^2 \d t &\leq 2(K+M_1)^2 \norm[{L^2\left([0,T],\sobsp\beta\right)}]{\*u_m}^2 + 2 \norm[{L^2([0,T],L^2)}]{\*f}^2 \\
        &\leq c'' \left(\norm{h}^2 + \norm[{L^\infty([0,T],L^2)}]{\*f}^2\right)
    \end{align}
    where we picked $c'' = 2\max\left\{(K+M_1)^2 c', T\right\}$.
\end{proof}

\begin{remark}
    In the case that $K > M_1$, in which case we may also allow $\beta = \frac12$, we get linear dependence of $C$ on $T$. If we track the constants in the above proof, we notice that the constant bounding the first term in \eqref{eq:energy_estimates} is 1, and the bound on the remaining two terms depends only linearly on $T$.
\end{remark}

Now, we are ready to show the existence of a solution to the full problem \eqref{eq:weak_FPDE}.

\begin{theorem}[Existence of a weak solution]\label{thm:existence}
    Under the conditions of Theorems \ref{thm:approx_solns} and \ref{thm:energy_estimates}, there is a weak solution $\*u \in C([0,T],L^2)$ to the PDE problem \eqref{eq:weak_FPDE}.
\end{theorem}

\begin{proof}
    \cref{thm:energy_estimates} shows that the sequences $\{\*u_m\}_{m=1}^\infty$ and $\{\*u_m^{(\alpha)}\}_{m=1}^\infty$ are bounded in \linebreak $L^2([0,T],\sobsp\beta)$ and $L^2([0,T],\sobsp{-\beta})$, respectively.
    Both spaces are Hilbert spaces, so by classical functional analysis, 
    there exists a subsequence $\{\*u_{m_l}\}_{l=1}^\infty \subset \{\*u_m\}_{m=1}^\infty$ such that
    \begin{align}
        \*u_{m_l} &\xrightharpoonup[]{} \*u\phantom{{}^{(\alpha)}} \qquad \text{ weakly in } L^2\left([0,T],\sobsp\beta\right) \label{eq:existence_weak_conv1} \\
        \*u_{m_l}^{(\alpha)} &\xrightharpoonup[]{} \*u^{(\alpha)} \qquad \text{ weakly in } L^2\left([0,T],\sobsp{-\beta}\right), \label{eq:existence_weak_conv2}
    \end{align}
    where $\*u \in L^2\left([0,T],\sobsp\beta\right)$ with $\*u^{(\alpha)} \in L^2\left([0,T],\sobsp{-\beta}\right)$.
    Note that $\*u_{m_l}^{(\alpha)}$ must converge to $\*u^{(\alpha)}$ and no other limit by \cref{prop:vprime_equals_u}.

    We have previously verified that $\*u(0) = h$  in \cref{lemma:limit_ic}.
    Furthermore, \cref{thm:cavalho_neto_fehlberg} establishes that $\*u$ is continuous in time in the sense that $\*u \in C([0,T],L^2)$.
    It therefore remains to show that $\*u$ and $\*u^{(\alpha)}$ satisfy the weak PDE \eqref{eq:weak_FPDE}.
    
    Consider an arbitrary function $\*v \in L^2([0,T],\sobsp\beta)$ of the form
    \begin{equation}\label{eq:existence_v_form}
        \*v(t) = \sum_{|k| \leq n} d^k(t) w_k.
    \end{equation}
    By taking appropriate linear combinations of \eqref{eq:approx_eqn} and integrating with respect to $t$, we have that for large enough $l$ such that $m_l \geq n$,
    \begin{equation}
        \int_0^T \inprod[{\sobsp{-\beta},\sobsp\beta}]{\*u_{m_l}^{(\alpha)}}{\*v} \d t + \int_0^T B[\*u_{m_l},\*v] \d t = \int_0^T \inprod{\*f}{\*v} \d t.
    \end{equation}
    The second term can be treated as a bounded linear functional acting on $\*u_{m_l}$, so we can pass to weak limits to obtain
    \begin{equation}\label{eq:existence_weak_eqn}
        \int_0^T \inprod[{\sobsp{-\beta},\sobsp\beta}]{\*u^{(\alpha)}}{\*v} \d t + \int_0^T B[\*u,\*v] \d t = \int_0^T \inprod{\*f}{\*v} \d t.
    \end{equation}
    Because functions $\*v$ of the form \eqref{eq:existence_v_form} are dense in $L^2([0,T],\sobsp\beta)$, \eqref{eq:existence_weak_eqn} holds for all $\*v \in L^2([0,T],\sobsp\beta)$.
    In particular, 
    \begin{equation}
        \inprod[{\sobsp{-\beta},\sobsp\beta}]{\*u^{(\alpha)}}{v} + B[\*u,v] = \inprod{\*f}{v}
    \end{equation}
    for any $v \in \sobsp\beta$ and almost every time $t \in [0,T]$, and hence we have established that $\*u$ is indeed the weak solution sought.
\end{proof}

\subsection{Uniqueness}
We now establish uniqueness; our approach is standard for linear evolutionary equations and can be compared directly with \cite[p381]{evans2010}.

\begin{theorem}[Uniqueness of the weak solution]\label{thm:uniqueness}
    There is a single weak solution to \eqref{eq:weak_FPDE}.
\end{theorem}

\begin{proof}
    Suppose that $\*u_1$ and $\*u_2$ are two weak solutions and define $\*u = \*u_1 - \*u_2$.
    Then $\*u$ satisfies \eqref{eq:weak_FPDE} with $h = 0$ and $\*f = 0$.
    In particular,
    \begin{equation*}
        \inprod[{\sobsp{-\beta},\sobsp\beta}]{\*u^{(\alpha)}}{\*u} + B[\*u,\*u] = 0
    \end{equation*}
    for almost every $t$.
    By \cref{thm:cavalho_neto_fehlberg} this implies
    \begin{equation*}
        \deriv[\alpha]{}{t} \norm{\*u}^2 + 2B[\*u,\*u] \leq 0.
    \end{equation*}
    Combining this with \eqref{eq:coercivity_bound} means that there exists $\gamma \geq 0$ such that
    \begin{equation*}
        \deriv[\alpha]{}{t} \norm{\*u}^2 \leq - 2 B[\*u,\*u] \leq \gamma \norm{\*u}^2.
    \end{equation*}
    Finally, we deduce from \cref{thm:gronwall} that $\norm{\*u}^2 \leq 0$, or in other words, 
    \begin{equation*}
        \*u = 0.\qedhere
    \end{equation*}
\end{proof}

\section{Regularity}\label{sec:regularity}

This section aims to show additional properties of the solution $\*u$ from \cref{thm:existence} that allow us to insert $\*u$ into the ``strong'' PDE.
The next three results will help in the proof of \cref{thm:regularity}.

In the section on numerical simulation of the equation studied, we will vary $\*V$. 
Thus, we need a bound that depends ``uniformly'' on $\*V$. 
This is the reason for introducing the constants $M_1$ before and $M_2$ in the following lemma.
\begin{lemma}\label{lemma:regularity_lemma1}
    Suppose $\beta > \frac12$ and $\*u\in \sobsp{2\beta}$ with $\*u^{(\alpha)} \in L^2$ satisfies
    \begin{equation*}
        \inprod{\*u^{(\alpha)}}{\flap\beta \*u} + B[\*u,\flap\beta \*u] = \inprod{\*f}{\flap\beta\*u}
    \end{equation*}
    for almost every $0 \leq t \leq T$.
    Assume $K > 0$, $\*f \in C([0,T],L^2)$ and $\Delta \*V \in {L^\infty([0,T],L^\infty)}$ and choose 
    \begin{equation*}
        M_2 \geq \max\left\{\norm[{L^\infty([0,T],L^\infty)}]{\Delta \*V}, \norm[{L^\infty([0,T],L^\infty)}]{\nabla \*V}\right\}.
    \end{equation*}
    Then there exists a constant $c(K,M_2)$ such that for almost every $t$
    \begin{equation}\label{eq:regularity_lemma1}
        \norm[H^{2\beta}]{\*u}^2 \leq c(K,M_2) \left(\norm{\*u^{(\alpha)}}^2 + \norm{\*u}^2 + \norm{\*f}^2\right).
    \end{equation}
\end{lemma}

\begin{proof}
    Given $0\leq t \leq T$ such that the assumption holds, we have
    \begin{equation*}
        K \inprod{\flap\beta \*u}{\flap\beta \*u} = \inprod{\nabla \cdot (\nabla \*V \*u) - \*u^{(\alpha)} + \*f}{\flap\beta \*u}.
    \end{equation*}
    Thus, using \eqref{eq:apx_Cauchy_Young_ineq} for $\eta > 0$ yields
    \begin{align*}
        K &\norm{\flap\beta \*u}^2\leq \frac1{4\eta} \norm{\nabla \cdot (\nabla \*V \*u) -\*u^{(\alpha)} + \*f}^2 + \eta \norm{\flap\beta \*u}^2 \\
        &\leq \frac1{4\eta} \left(\norm{\*u^{(\alpha)}}^2 + \norm{\Delta \*V \*u}^2 + \norm{\nabla \*V \nabla \*u}^2 + \norm{\*f}^2\right) + \eta \norm{\flap\beta \*u}^2 \\
        &\leq \frac1{4\eta}\left(\norm{\*u^{(\alpha)}}^2 + \norm[L^\infty]{\Delta \*V}^2 \norm{\*u}^2 + \norm[L^\infty]{\nabla \*V}^2 \norm{\nabla \*u}^2 + \norm{\*f}^2\right) + \eta \norm{\flap\beta \*u}^2 \\
        &\leq \frac1{4\eta}\norm{\*u^{(\alpha)}}^2 + \frac{M_2^2}{4\eta} \left(\norm{\*u}^2 + \norm{\nabla \*u}^2\right) + \frac1{4\eta} \norm{\*f}^2 + \eta \norm{\flap\beta \*u}^2.
    \end{align*}
    Recalling $\norm{\nabla \*u} = \norm{\flap{1/2} \*u}$ from \eqref{eq:H1_norms_are_equivalent} and using \cref{lemma:youngtype_ineq} (note that $\beta > \frac12$) for $\frac{4 \eta^2}{M_2^2} > 0$ we can find an appropriate constant $c(M_2,\eta)$ such that
    \begin{equation*}
        K \norm{\flap\beta \*u}^2 \leq \frac1{4\eta} \norm{\*u^{(\alpha)}}^2 + c(M_2,\eta) \norm{\*u}^2 + \frac1{4\eta} \norm{\*f}^2 + 2 \eta \norm{\flap\beta \*u}^2.
    \end{equation*}
    Now choosing $\eta > 0$ sufficiently small implies the result.
\end{proof}

\begin{lemma}\label{lemma:regularity_lemma2}
    Given $\varepsilon > 0$, with the same assumptions as in \cref{lemma:regularity_lemma1}, there is a constant $c(K,M_2,\varepsilon)$ such that 
    \begin{equation}\label{eq:regularity_lemma2}
        \int_{\T^d} \left|\nabla \cdot (\nabla \*V \*u) \*u^{(\alpha)} \right| \d x \leq \varepsilon \left(\norm{\*u^{(\alpha)}}^2 + \norm{\*f}^2 \right) + c(K,M_2,\varepsilon) \norm{\*u}^2
    \end{equation}
    for almost every $t$.
\end{lemma}

\begin{proof}
    For any $\eta > 0$, use \eqref{eq:apx_Cauchy_Young_ineq} to get
    \begin{align*}
        \int_{\T^d} \left|\nabla \cdot (\nabla \*V \*u) \*u^{(\alpha)} \right| \d x &\leq M_2 \int_{\T^d} \left(|\*u| + |\nabla \*u| \right) |\*u^{(\alpha)}| \d x \\
        &\leq \frac {M_2}{4\eta} \int_{\T^d} |\*u|^2 + |\nabla \*u|^2 \d x + \eta \int_{\T^d} |\*u^{(\alpha)}|^2 \d x \\
        &\leq \frac {M_2}{4\eta} \norm{\flap{1/2} \*u}^2 + \frac {M_2}{4\eta} \norm{\*u}^2 + \eta \norm{\*u^{(\alpha)}}^2.
    \end{align*}
    Now we can apply \cref{lemma:youngtype_ineq} with $\frac{4\eta^2}{c(K,M_2) M_2}$ to obtain
    \begin{equation*}
        \int_{\T^d} \left|\nabla \cdot (\nabla \*V \*u) \*u^{(\alpha)} \right| \d x \leq \frac\eta{c(K,M_2)} \norm{\flap\beta \*u}^2 + c(K,M_2,\eta) \norm{\*u }^2  + \eta \norm{\*u^{(\alpha)}}^2
    \end{equation*}
    for some constant $c(K,M_2,\eta)$.
    Finally, applying \cref{lemma:regularity_lemma1} and picking $\eta > 0$ small enough we conclude that \eqref{eq:regularity_lemma2} holds.
\end{proof}

We need one final ingredient to prove some extra regularity for our solution.

\begin{proposition}[Time Poincar\'e-type inequality for the fractional integral operator]\label{prop:integral_poincare}
    For $\alpha \in (0,1]$ and a  nonnegative function $g$, there exists a constant $c(T,\alpha)$ depending on $T$ and $\alpha$ such that
    \begin{equation}
        \int_0^t g(\tau) \d \tau \leq c(T, \alpha) I^\alpha g(t).
    \end{equation}
\end{proposition}

\begin{proof}
    Using H\"older's inequality, we compute
    \begin{align*}
        \int_0^t g(\tau) \d \tau &= \int_0^t (t-\tau)^{1-\alpha} (t-\tau)^{\alpha-1} g(\tau) \d \tau \\
        &\leq t^{1-\alpha} \int_0^t (t-\tau)^{\alpha-1} g(\tau) \d \tau \\
        &= \Gamma(\alpha) t^{1-\alpha} I^\alpha g(t) \leq \Gamma(\alpha) T^{1-\alpha} I^\alpha g(t).\qedhere
    \end{align*}
\end{proof}

We are now in position to prove higher regularity of the weak solution.

\begin{theorem}[Improved regularity of weak solutions]\label{thm:regularity}
    Suppose that $\beta > \frac12$; $K > 0$; $\*V \in C([0,T], H^1) \cap L^\infty([0,T],H^3)$; $\*f \in L^\infty([0,T],L^2)$ and $h\in \sobsp\beta$. Choose $M = \max \left\{M_1,M_2\right\}$. Assume that $\*u \in L^2([0,T],\sobsp\beta)$ with $\*u^{(\alpha)} \in L^2([0,T],\sobsp{-\beta})$ satisfies
    \begin{equation*}
        \inprod[\sobsp{-\beta},\sobsp\beta]{\*u^{(\alpha)}}v + B[\*u,v] = \inprod{\*f}{v}
    \end{equation*}
    for any $v \in \sobsp\beta$ and almost every $0\leq t \leq T$, with the initial condition $\*u(0) = h$.
    Then we have the estimate
    \begin{equation} \label{eq:regularity_thm}
        \norm[{L^2([0,T],\sobsp{2\beta})}] {\*u}^2 + \norm[{L^\infty([0,T],\sobsp\beta)}]{\*u} + \norm[{L^2([0,T],L^2)}]{\*u^{(\alpha)}} \leq C \left(\norm[H^\beta] h + \norm[{L^\infty([0,T],L^2)}]{\*f}\right)
    \end{equation}
    for some constant $C$ depending on $K$, $M$, $T$, $\alpha$ and $\beta$, 
    which immediately implies that, in fact,
    \begin{equation*}
        \*u \in L^2([0,T],\sobsp{2\beta}) \cap L^\infty([0,T],\sobsp\beta) \text{ and } \*u^{(\alpha)} \in L^2([0,T],L^2).
    \end{equation*}
\end{theorem}

\begin{proof}
    For each $m$, take a linear combination of \eqref{eq:approx_eqn} with coefficients $\deriv[\alpha]{}t d_m^k(t)$ for each $|k| \leq m$
    \begin{equation*}
        \inprod{\*u_m^{(\alpha)}}{\*u_m^{(\alpha)}} + B[\*u_m, \*u_m^{(\alpha)}] = \inprod{\*f}{\*u_m^{(\alpha)}}.
    \end{equation*}
    Using \cref{thm:cavalho_neto_fehlberg}, \cref{lemma:regularity_lemma2} and \eqref{eq:apx_Cauchy_Young_ineq} we get the chain of inequalities
    \begin{align*}
        \norm{\*u_m^{(\alpha)}}^2 + \frac12\deriv[\alpha]{}{t} \norm{\flap{\beta/2} \*u_m}^2 &\leq \inprod{\*u_m^{(\alpha)}}{\*u_m^{(\alpha)}} + B[\*u_m, \*u_m^{(\alpha)}] + \int_{\T^d} \nabla \cdot (\nabla \*V \*u_m) \*u_m^{(\alpha)} \d x \\
        &\leq \inprod{\*f}{\*u_m^{(\alpha)}} + \varepsilon \left(\norm{\*u_m^{(\alpha)}}^2 + \norm{\*f}^2\right) + c(K,M_2,\varepsilon) \norm{\*u_m}^2 \\
        &\leq 2\varepsilon\norm{\*u_m^{(\alpha)}}^2 + \left(\varepsilon + \frac1{4\varepsilon}\right) \norm{\*f}^2 + c(K,M_2,\varepsilon) \norm{\*u_m}^2
    \end{align*}
    for any $\varepsilon>0$. To justify this, observe that if $\*u_m \in \mspan\left(\{w_k\}_{|k|\leq m}\right)$, then so is $\flap\beta \*u_m$, and the conditions stated in \cref{lemma:regularity_lemma1} are satisfied. Choosing $\varepsilon = \frac14$,
    \begin{align}
        \norm{\*u_m^{(\alpha)}}^2 + \deriv[\alpha]{}{t} \norm{\flap{\beta/2} \*u_m}^2 &\leq c(K,M_2) \norm{\*u_m}^2 + \frac54 \norm{\*f}^2 \notag\\
        &\leq c(K,M_2,T) \left(\norm{h}^2 + \norm[{L^\infty([0,T],L^2)}]{\*f}^2\right) \label{eq:regularity_proof}
    \end{align}
    for some constants $c(K,M_2)$ and $c(K,M_2,T)$, using from \cref{thm:energy_estimates} that \begin{equation*}
        \norm{\*u_m}^2 \leq C \left(\norm{h}^2 +\norm[{L^\infty([0,T],L^2)}]{\*f}^2\right).
    \end{equation*}

    Next, apply $I^\alpha$ to both sides,observing once more that this preserves inequalities, and use \cref{prop:integral_poincare} to estimate the first term. Then
    \begin{equation*}
        \int_0^t \norm{\*u_m^{(\alpha)}}^2 \d t + \norm{\flap{\beta/2} \*u_m(t)}^2 - \norm{\flap{\beta/2} \*u_m(0)}^2 \leq \tilde c(K,M,T) \left(\norm{h}^2 + \norm[{L^\infty([0,T],L^2)}]{\*f}^2\right)
    \end{equation*}
    for almost every $t$.
    By Bessel's inequality, $\norm{\flap{\beta/2}\*u_m(0)}^2 \leq \norm{\flap{\beta/2}h}^2$. Furthermore, \eqref{eq:regularity_lemma1} provides a bound of $\norm[H^{2\beta}]{\*u_m}$ in terms of $\norm{\*u_m^{(\alpha)}}$. Hence, implementing this bound and taking the (essential) supremum over $0 \leq t \leq T$ of the left-hand side we deduce
    \begin{equation*}
        \norm[{L^2([0,T],\sobsp{2\beta})}]{\*u_m}^2 + \norm[{L^2([0,T],L^2)}]{\*u_m^{(\alpha)}}^2 + \norm[{L^\infty([0,T],\sobsp\beta)}]{\*u_m}^2 \leq \tilde C \left(\norm[H^\beta]{h}^2 + \norm[{L^\infty([0,T],L^2)}]{\*f}^2\right).
    \end{equation*}
    The first two terms are natural norms on Hilbert spaces, so by standard functional analysis, we can pass to weak limits and preserve the inequality.
    For the final term, setting $m = m_l$ and letting $l\to \infty$ preserves the inequality by Problem 6 of Chapter 7 in \cite{evans2010}.
    Thus, we have shown \eqref{eq:regularity_thm} and the proof is complete.
\end{proof}

\begin{remark}
    We note that we could certainly obtain yet further improved regularity result under stronger assumptions on the initial condition, source term and potential, but we do not pursue this further here. Our main motivation for establishing this level of regularity is that it allows us to establish estimates on an appropriate numerical scheme discussed below.
\end{remark}

\section{Numerical scheme}\label{sec:numerics}

We can use the estimates established in \cref{thm:regularity} to demonstrate the convergence of numerical approximate solutions with a natural scheme in which we uniformly truncate the Fourier coefficients of all functions involved in the equation.

Consider as before the equation 
\begin{equation*}
    \pd[\alpha]{}{t} u(x,t) + K \flap{\beta} u(x,t) - \nabla \cdot (\nabla V u) = f
\end{equation*}
subject to initial conditions $u(\cdot,0) = h \in \sobsp\beta$, a given potential $V(\cdot,t) \in H^3 \subset C^3$ and a source term $f(\cdot,t) \in L^2$ for each $t$ such that $\nabla V$ and $f$ are continuous in time.

Fix a positive integer $m$. 
For a numerical scheme, we want to work in the finite-dimensional subspace $\mspan (\{w_k\}_{|k|\leq m})$.
Thus, define
\begin{equation*}
    V_m(x) = \sum_{|k| \leq m} \ft V(k) w_k(x)
\end{equation*}
to be the potential obtained by truncating the Fourier series of $V$.
For a fixed $m$, it is now possible to compute the solution $\tilde u_m(x,t)$ to the system of equations
\begin{equation}\label{eq:numerics_trunc_eq}
    \inprod{\pd[\alpha]{}{t}\tilde u_m}{w_k} + K \inprod{\flap{-\beta} \tilde u_m}{w_k} - \inprod{\nabla \cdot (\nabla V_m \tilde u_m)}{w_k } = \inprod{f}{w_k}
\end{equation}
with initial conditions $\inprod{\tilde u_m(\cdot,0)}{w_k} = \inprod{h}{w_k}$ for $|k| \leq m$.
Note that truncating the Fourier series of $f$ in the same fashion as we truncate the Fourier series of $V$ is always possible since this does not affect the equation. In this section, we will show that
\begin{equation*}
    \tilde u_m \to u \quad \text{ as } m \to \infty
\end{equation*}
uniformly in time (on an interval $[0,T]$) and $L^2$ in space, where $u$ is the full (weak) solution. Recall that the approximate solutions $u_m$ defined by $u_m(x,t) := \*u_m(t)(x)$ obtained in \cref{thm:approx_solns} solve the system
\begin{equation}\label{eq:numerics_full_eq}
    \inprod{\pd[\alpha]{}{t}u_m}{w_k} + K \inprod{\flap{-\beta} u_m}{w_k} - \inprod{\nabla \cdot (\nabla V u_m)}{w_k} = \inprod{f}{w_k}
\end{equation}
for $|k| \leq m$, subject to the same initial conditions.

First, we want to justify that truncating the Fourier series of $V$ and finding solutions to \eqref{eq:numerics_trunc_eq} instead of \eqref{eq:numerics_full_eq} still approximates the original weak equation.
Subtracting \eqref{eq:numerics_trunc_eq} from \eqref{eq:numerics_full_eq} yields
\begin{equation}\label{eq:numerics1}
    \inprod{\pd[\alpha]{}{t}\tilde U_m}{w_k} + B[\tilde U_m,w_k] + \inprod{\nabla (V-V_m) \tilde u_m}{\nabla w_k} = 0,
\end{equation}
where $\tilde U_m = u_m - \tilde u_m$ is the residual.
Now we can proceed similarly to the first few steps of the proof of \cref{thm:energy_estimates}. Taking a linear combination over $w_k$ in \eqref{eq:numerics1} using appropriate Fourier coefficients, we obtain
\begin{equation*}
    \inprod{\pd[\alpha]{}{t}\tilde U_m}{\tilde U_m} + B[\tilde U_m,\tilde U_m] + \inprod{\nabla (V-V_m) \tilde u_m}{\nabla \tilde U_m} = 0.
\end{equation*}
From \cref{prop:weak_bound_forcing_term}, we have that for fixed $t$,
\begin{equation}
    - \inprod{\nabla (V-V_m) \tilde u_m}{\nabla \tilde U_m} \leq c(V-V_m) \norm{\flap{1/4} \tilde u_m} \norm{\flap{1/4}\tilde U_m}.
\end{equation}
Now notice that $\tilde u_m$ and $u_m \in L^\infty([0,T],\sobsp\beta)$, but the potential $V_m$ in \eqref{eq:numerics_trunc_eq} depends on $m$. However, because $V(\cdot,t) \in H^3$ we have that
\begin{equation*}
    \norm[l^1]{\longft{\flap{3/4} (V-V_m)}} \to 0 \text{ as } m \to \infty.
\end{equation*}
In particular, this means we can choose the bounds $M_1$ and $M_2$ for $V_m$ and $V$ in \cref{thm:regularity} uniformly in $m$ and $t$. 
Furthermore, we also directly obtain that $c(V-V_m) \to 0$ uniformly in $t$ as $m\to \infty$.
Therefore, there exists a nonnegative sequence $\varepsilon_m \to 0$ as $m\to \infty$ such that
\begin{equation}\label{eq:numerics_diff_ineq}
    \inprod{\pd[\alpha]{}{t}\tilde U_m}{\tilde U_m} + B[\tilde U_m,\tilde U_m] \leq \varepsilon_m.
\end{equation}
Using \cref{thm:cavalho_neto_fehlberg}, 
\begin{equation*}
    \frac12 \deriv[\alpha]{}t \left(\norm{\tilde U_m}^2\right) \leq \inprod{\tilde U_m^{(\alpha)}}{\tilde U_m} \leq - B[\tilde U_m,\tilde U_m] + \varepsilon_m \leq \gamma_2 \norm{\tilde U_m}^2 + \varepsilon_m.
\end{equation*}
Now we can use \hyperref[thm:gronwall]{Gr\"onwall's inequality} to estimate
\begin{equation*}
    \norm{\tilde U_m}^2 \leq \left( I^\alpha(2\varepsilon_m) + \norm{\tilde U_m(0)}^2\right) E_\alpha(2 \gamma_2 t^\alpha).
\end{equation*}
Noting that $I^\alpha(2\varepsilon_m) = \frac{2\varepsilon_m}{\Gamma(\alpha+1)} t^\alpha$ and $\norm{\tilde U_m(0)}^2 = 0$, we conclude that there is a constant $\tilde C$ independent of $m$ such that
\begin{equation*}
    \norm{\tilde U_m}^2 \leq \tilde C \varepsilon_m.
\end{equation*}
Thus, 
\begin{equation}\label{eq:numerics_justification}
    \norm{u_m - \tilde u_m} \to 0 \quad \text{ uniformly in $t$ as } m\to \infty,
\end{equation}
which justifies that we may truncate the potential $V$ (and the source term $f$) in a numerical scheme.

Secondly, let us return to the sequence $\*u_m$ constructed in \cref{thm:approx_solns}.
So far we know that a subsequence of $\*u_m$ converges weakly in the appropriate Bochner space to $\*u$.
Define the residual $\*U_m = \*u - \*u_m$, which satisfies
\begin{equation*}
    \inprod{\*U_m^{(\alpha)}}{v} + B[\*U_m, v] = \inprod{\*f-\*f_m}{v}
\end{equation*}
for all $v \in \sobsp\beta$, subject to the initial condition
\begin{equation*}
    \*U_m(0) = h-h_m.
\end{equation*}
Here, $\*f_m$ and $h_m$ are the projections of $\*f$ and $h$ onto the first $m$ Fourier modes, respectively.
Thus, by \cref{thm:regularity}, we have that 
\begin{equation*}
    \esssup_{0\leq t \leq T} \norm[H^\beta]{\*U_m(t)} \leq C \left( \norm[H^\beta]{h-h_m} + \norm[{L^\infty([0,T],\sobsp\beta)}] { \*f- \*f_m}\right).
\end{equation*}
Hence, for sufficiently regular initial condition and source term, we obtain that 
\begin{equation*}
    \norm[H^\beta]{\*u(t) - \*u_m(t)} \to 0 \quad \text { as } m \to \infty
\end{equation*}
uniformly for almost every $t \in [0,T]$.
Then, combined with \eqref{eq:numerics_justification}, we finally have that the numerical solutions computed in \eqref{eq:numerics_trunc_eq} converge uniformly in time and $L^2$ in space to the true weak solution $\*u$. In other words,
\begin{equation*}
    \tilde u_m(\cdot, t) \overset{L^2}{\to} \*u(t) \quad \text{ as } m \to \infty
\end{equation*}
uniformly in time as long as $T$ is finite.

\section{Conclusion}
The main contribution of this work is to show existence, uniqueness and Sobolev regularity of a weak solution to the fractional differential equation of parabolic type
\begin{equation*}
    \pd[\alpha]{u}{t} - K \flap{\beta} u - \nabla \cdot(\nabla V u) = f
\end{equation*}
on a $d$-dimensional torus $\T^d$ and a time interval $[0,T]$.
Intuitively, the term containing the fractional Laplacian models a diffusive process of a particle over time according to a $2\beta$-stable L\'evy-process (see \cref{fig:Levy_process}).
The Caputo fractional derivative with respect to time encapsulates the fact that the particle's movement at every time also depends on previous states.
The term containing the potential $V$ pushes the particle towards the minima of $V$.
Finally, the source term $f$ adds sources and sinks to the equation.

More precisely, we showed that if $\beta > \frac12$, then there is a unique $u: T^d \times [0,T] \to \C$ such that
\begin{equation*}\begin{aligned}
    \Bigl\langle\pd[\alpha]{u}{t}, v \Bigr\rangle + K \inprod[]{ \flap{\beta/2} u}{ \flap{\beta/2} v} + \inprod[] {\left(u \nabla V \right)} { \nabla v } &= \inprod[]{f}{v}  &&\text{ on } \T^d \text{ for a.e. } t\in (0,T), \\
    u &= h  &&\text{ on } \T^d \text{ for } t = 0
\end{aligned}\end{equation*}
for any $v \in \sobsp\beta$. 
Here, $K>0$ is a constant, $h : \T^d \to \C$ is the initial condition, and $V, f: \T^d \times [0,T] \to \C$ are given functions that are sufficiently regular (precise conditions can be found in Theorems \ref{thm:existence} and \ref{thm:regularity}).

Our existence proof uses Galerkin approximation in finite-dimensional subspaces of $L^2(\C)$ spanned by elements of the orthonormal basis $w_k(x) := \frac1{(2\pi)^{d/2}} e^{ik\cdot x}$.
In particular, it is easier to numerically find a solution $\tilde u_m (x,t)$ for fixed $m$ to the set of $(2m+1)^d$ ordinary fractional differential equations (one for each $|k| \leq m$) given by
\begin{equation} \begin{aligned} \label{eq:numerics_trunc_eq_conclusion}
    \inprod{\pd[\alpha]{}{t}\tilde u_m}{w_k} + K \inprod{\flap{-\beta} \tilde u_m}{w_k} - \inprod{\nabla \cdot (\nabla V_m \tilde u_m)}{w_k } &= \inprod{f}{w_k} \\
    \inprod{\tilde u_m(\cdot,0)}{w_k} &= \inprod{h}{w_k}
\end{aligned}.\end{equation}
In \cref{sec:numerics} we showed that such a solution $\tilde u_m$ to \eqref{eq:numerics_trunc_eq_conclusion} is a good approximation to the true solution $u$ in the sense that
\begin{equation*}
    \tilde u_m(\cdot, t) \overset{L^2}{\to} u(\cdot, t) \quad \text{ as } m \to \infty
\end{equation*}
uniformly on $[0,T]$.

The system \eqref{eq:numerics_trunc_eq_conclusion} boils down to a coupled ``$\alpha$th-order'' system of ODEs that can be solved by computing the Mittag-Leffler function of the matrix corresponding to the ODE system.
This algorithm was used to produce the plot in \cref{fig:3d_sol}.

It is interesting to see that in \cref{fig:final_time_sol} the numerical result for $\beta = 0.4 < \frac12$ looks perfectly reasonable.
In fact, the numerical scheme breaks only for $\alpha$ or $\beta$ close to 0, suggesting that $\beta > \frac12$ is not a necessary condition.
However, note that if $\beta < \frac12$ the potential term is now the term with the highest-order derivative, and the equation is no longer parabolic, which critically allowed us to use \cref{prop:weak_bound_forcing_term}. 

Requiring the spatial domain to be a torus in order to work with Fourier series is one of the most significant restrictions in this work. In the future, it would be interesting to consider the equation on other spatial domains, but such results involve subtleties due to the need to interpret boundary conditions nonlocally.

The regularity result proved in \cref{thm:regularity} is relatively weak, but it is to be expected that one could probably improve upon the solution's regularity under possible stronger assumptions on the initial condition, source term and potential. Finally, one can further investigate fractional diffusion equations where the term involving the potential takes on a more general form.

\subsection*{Acknowledgements}
The authors would like to acknowledge the support of the University of Warwick's Undergraduate Research Support Scheme, which provided financial support for MR to work on this project.

\subsection*{Authors' statements}
The authors declare that there is no conflict of interest arising from the publication of this work. For the purpose of open access, the authors have applied a Creative Commons Attribution (CC-BY) licence to any Author Accepted Manuscript version arising from this work.

\bibliography{references.bib}

\end{document}